\documentclass[10pt]{amsart}
\usepackage{amssymb, amsthm, amsmath, mathrsfs, yhmath, slashed, wasysym, epstopdf, graphicx, bbm, color}

\newtheorem{theorem}{Theorem}
\newtheorem{lemma}{Lemma}

\newtheorem{proposition}{Proposition}

\theoremstyle{definition}

\theoremstyle{definition}
\newtheorem{remark*}{Remark}

\theoremstyle{definition}

\newcommand{ \rn }[1]{\mathbb{R}^{#1}}
\title{The Broken Ray Transform on the Square}
\date{}
\author[M. Hubenthal]{Mark Hubenthal}
\address{Department of Mathematics and Statistics, University of Jyv\"askyl\"a}
\email{john.m.hubenthal@jyu.fi}
\subjclass[2010]{45Q05, 47G30, 53C65}
\keywords{inverse problems, integral geometry, microlocal analysis}

\begin{document}
\maketitle
\begin{abstract}We study a particular broken ray transform on the Euclidean unit square and establish injectivity and stability for $C_{0}^{2}$ perturbations of the constant unit weight. Given an open subset $E$ of the boundary, we measure the attenuation of all broken rays starting and ending at $E$ with the standard optical reflection rule. Using the analytic microlocal approach of Frigyik, Stefanov, and Uhlmann for the X-ray transform on generic families of curves, we show injectivity via a path unfolding argument under suitable conditions on the available broken rays. Then we show that with a suitable decomposition of the measurement operator via smooth cutoff functions, the associated normal operator is a classical pseudo differential operator of order $-1$ plus a smoothing term with $C_{0}^{\infty}$ Schwartz kernel, which leads to the desired result.\end{abstract}

\section{Introduction \label{sec:intro}}
In this article, we are interested in the questions of injectivity and stability for a particular kind of Euclidean broken ray transform on the unit square in $2$ dimensions. Much work has already been done previously with regard to the X-ray transform and its attenuated analogue. In \cite{novikov, novikov2}, Novikov presents an explicit reconstruction formula for the attenuated ray transform in $2$ dimensions and derives specific conditions for the range in terms of certain integral relations. Bal subsequently derived in \cite{balradon} a specific reconstruction scheme based on the inversion formula of Novikov, which exploits some redundancies in the data and even considers the case of an angularly varying source. Frigyik, Stefanov, and Uhlmann showed injectivity and stability for the X-ray transform over generic families of curves on smooth manifolds in \cite{xraygeneric} where the attenuation is close to real analytic. Previously, Stefanov and Uhlmann in \cite{xraytensor} analyzed the linearized boundary rigidity problem which leads to the geodesic X-ray transform of symmetric 2-tensor fields. They showed in particular that the corresponding normal operator is a pseudodifferential operator of order $-1$ if the manifold is simple.  We also refer to (\cite{boman, finch1, finch2, natterer}) for further background and results on X-ray tomography. 

Recall the unattenuated Euclidean X-ray transform of an unknown function $f \in L^{2}(\Omega)$ supported on some bounded domain $\Omega$, which is given by
\begin{equation*}
Xf(x,\theta) = \int_{\mathbb{R}}f(x+t\theta)\,dt, \quad (x,\theta) \in \Omega \times \mathbb{S}^{n-1}.
\end{equation*}
Roughly speaking, the main condition needed for injectivity of $X$ is that the union of the conormal bundles over all available lines must span the cotangent space of the domain (see \cite{xraygeneric}). In the case of the geodesic X-ray transform, this condition holds if all geodesics have no conjugate points, which in turn is true if one assumes that the underlying manifold is simple. But of course such a condition trivially holds for lines in the Euclidean setting. We do note, however, that such a condition was relaxed more recently when Stefanov and Uhlmann considered the case of microlocal invertibility for the geodesic ray transform with possible isolated conjugate points along certain geodesics, i.e. in the presence of fold caustics (\cite{causticxray}).  

The primary tool used to prove injectivity in \cite{xraygeneric} for the generalized X-ray transform, and more specifically to recover the singularities of an unknown interior distribution $f$, is analytic microlocal analysis. While the standard $C^{\infty}$ pseudodifferential calculus allows one to use the normal operator to determine $f$ modulo a smooth remainder, the analytic wavefront set is a more deterministic object. For if an unknown function $f$ supported in $\Omega$ satisfies $WF_{A}(f) \cap T^{*}\Omega = \emptyset$, then $f$ must be entire and compactly supported, and hence must vanish identically. We employ such ideas heavily in this work. However, it should be stated that other results pertaining to the attenuated ray transform or geodesic ray transform have utilized different methods.

The framework of this paper is of a relatively simple geometric nature, since few results have yet been attained for the problem we will consider. Roughly speaking, if $(\Omega, g)$ is a compact Riemannian manifold with appropriate conditions, and $E \subset \partial \Omega$ is open, the question is the following: do the integrals of an unknown function $f$ along all broken rays starting and ending on $E$, and which reflect only on $\partial \Omega \setminus E$, uniquely determine $f$? Unfortunately, reflection tomography of this nature for non-Euclidean metrics has not yet been treated in general. The technique of \cite{xraygeneric} relies on constructing local coordinates about some fixed geodesic $\gamma_{0}$ in order to show that the normal operator has nice microlocal properties. But, if one were to reflect from the boundary the family of geodesics associated with such coordinates, one no longer has local coordinates around the corresponding reflected segment of $\gamma_{0}$. As such, we will focus just on the Euclidean case, and in particular we will consider the unit square in $\rn{2}$. It is the author's belief, however, that the results herein are generalizable to at least to convex domains in $\rn{n}$ with piecewise flat reflective boundaries.

Although this problem is relatively new in its consideration, there are some previously existing results which are somewhat related. In particular, Eskin proved a uniqueness result for $C^{\infty}$ unknown functions in $2$-dimensions given knowledge of broken rays that reflect off of known convex obstacles in the domain's interior (\cite{eskin}). There the weight function is taken to be a smooth electric potential $V(x)$. One of the main tools used in the proof is a differential identity of Mukhometov (\cite{mukhometov}), specific to two dimensions, which involves the integrals of the potential along broken rays. Furthermore, the recent work of Ilmavirta in \cite{joonas} proves an injectivity result on the disk for open subsets $E$ of the boundary when the unknown function $f$ is uniformly quasianalytic in the angular variable (when written in polar coordinates). 

Finally, we mention some recent work of Salo and Kenig, which provides explicit motivation for the consideration of such a broken ray transform in the first place (\cite{salokenig}). There the authors consider the recovery of a Schr\"odinger potential on a particular class of compact manifolds from partial Cauchy data, deemed \textit{conformally transversally anisotropic}. This of course directly relates to the Calder\'on problem with partial data. In short, they consider the problem of recovering an unknown potential $q \in L^{\infty}(\Omega)$ given the partial Cauchy data set
\begin{align*}
\mathcal{C}_{q}^{\Gamma_{D}, \Gamma_{N}} & = \{ (u \vert_{\Gamma_{D}}, \partial_{\nu}u \vert_{\Gamma_{N}}) \, | \, (-\Delta + q)u = 0 \mathrm{ in }\Omega,\, u \in H_{\Delta}(\Omega),\\
& \qquad \qquad \qquad \qquad \quad \mathrm{supp}(u \vert_{\partial \Omega}) \subset \Gamma_{D} \}.
\end{align*}
Here $H_{\Delta}(\Omega) = \{ u \in L^{2}(\Omega) \, | \, \Delta u \in L^{2}(\Omega)\}$ and $\Gamma_{D}, \Gamma_{N}$ are two (possibly very small) open subsets of $\partial \Omega$. One very important point to note is that, differing from previous results for the Calder\'on problem with partial data, the authors here do not assume much about where the Neumann and Dirichlet data are accessible. Indeed, although numerous partial data results are known for the Schr\"odinger equation, each of them have to make inequivalent assumptions about the size of the accessible and inaccessible sets on the boundary. The authors then consider the case when $\Omega$ is a Riemannian manifold $(M,g) \Subset (\mathbb{R} \times \widetilde{M}_{0}, g)$ where $g = c(e \oplus g_{0})$, $(\widetilde{M}_{0}, g_{0})$ is some compact $(n-1)$-dimensional simple manifold with boundary, and $e$ is the euclidean metric on the real line. There is also one further technical assumption that
\begin{equation*}
(M,g) \subset (\mathbb{R} \times M_{0}, g) \Subset (\mathbb{R} \times \widehat{M}_{0}, g)
\end{equation*}
where $M_{0}$ is a compact $(n-1)$-dimensional manifold with smooth boundary and
\begin{equation*}
\partial M \cap (\mathbb{R} \times \partial M_{0}) \neq \emptyset.
\end{equation*}

One then has the limiting Carleman weight $\phi(x) = x_{1}$ to work with, which allows for decomposing the boundary $\partial M$ as the disjoint union
\begin{equation*}
\partial M = \partial M_{+} \cup \partial M_{-} \cup \partial M_{\mathrm{tan}},
\end{equation*}
where
\begin{align*}
\partial M_{\pm} & = \{ x \in \partial M \, | \, \pm \partial_{\nu} \phi(x) > 0\},\\
\partial M_{\mathrm{tan}} & = \{x \in \partial M \, | \, \partial_{\nu}\phi(x) = 0 \}.
\end{align*}
The authors of \cite{salokenig} then prove the following: let $\Gamma_{i} \subset \partial M_{\mathrm{tan}}$ such that there exists a nonempty open subset $E$ of $\partial M_{0}$, and suppose $\Gamma_{D} = \Gamma_{N} = \Gamma$ for some neighborhood $\Gamma$ in $\partial M$ of the set $\overline{\partial M_{+} \cup \partial M_{-} \cup \Gamma_{a}}$ for $\Gamma_{a} = \partial M_{\mathrm{tan}} \setminus \Gamma_{i}$. If $q_{1},q_{2}$ are two potentials such that $\mathcal{C}_{g,q_{1}}^{\Gamma_{D},\Gamma_{N}} = \mathcal{C}_{g,q_{2}}^{\Gamma_{D},\Gamma_{N}}$, then
\begin{equation*}
\int_{0}^{L}e^{-2\lambda t}( c(q_{1} - q_{2})\,{\hat{}}\,(2\lambda, \gamma(t))\,dt = 0
\end{equation*}
for all nontangential broken rays $\gamma: [0,L] \to M_{0}$ with endpoints in $E$ and for any $\lambda \in \mathbb{R}$. Here $( \cdot )\,{\hat{}}$ denotes the Fourier transform with respect to $x_{1}$.
Thus, the unique determination of the potential from such partial data could be resolved if one can show injectivity of the broken ray transform considered here. Ultimately, the big hope is that the collection of currently existing partial data results for the Calder\'on problem can be more unified, perhaps by using this (or some other) new approach to the problem.

The structure of this paper goes as follows. In \S \ref{sec:mainresults} we describe the problem and notations and then state the main results. In order to simplify the presentation, we start with the unattenuated case which is all that's needed to state the injectivity result. We subsequently introduce the broken ray transform with general attenuation and state a stability result for attenuation coefficients close to $0$ in $C^{2}$. \S \ref{sec:squareinjective} applies the microlocal ideas of \cite{xraygeneric} together with the path unfolding technique to prove injectivity with no attenuation. In \S \ref{sec:stability} we show the details required to prove the stability estimate of the inverse problem. To this end, \S \ref{sec:normaloperatorcompute}, \ref{sec:normalsimplify} and \ref{sec:corners} all deal with computations of the localized normal operator for different choices of neighborhoods of broken rays. Finally, we extend the stability estimate and injectivity to $C^{2}$ perturbations of the attenuation $\sigma$ in \S \ref{sec:reducingsmoothness}.

\subsection*{Acknowledgements}Support by the Institut Mittag-Leffler (Djursholm, Sweden) is gratefully acknowledged, as this work was completed there. This research was also conducted with partial support from the Academy of Finland while at the University of Jyv\"askyl\"a. The author would like to thank Mikko Salo for introducing him to
this problem and for providing helpful comments and feedback.

\section{Statement of Main Results \label{sec:mainresults}}
We will state the problem setup and definition in the case of a general domain $\Omega$ with smooth boundary. However, the main theorems themselves assume further that $\Omega$ is the square. Let $\Omega$ be a smooth bounded domain in $\rn{n}$ and define $\Gamma_{\pm} = \{(x,\theta) \in \partial \Omega \times \mathbb{S}^{n-1} \, | \, \pm \nu(x) \cdot \theta >0\}$ as the set of outgoing and ingoing unit vectors on $\partial \Omega$, respectively. Here $\nu(x)$ is the outward unit normal vector to $\partial \Omega$. We also define, for a general open subset $E \subset \partial \Omega$, the set
\begin{equation}
\Gamma_{\pm}(E) := \{(x,\theta) \in S\Omega \, | \, x \in E, \pm \theta \cdot \nu(x) > 0\}.
\end{equation}
Assume that $\sigma \in C^{\infty}(\Omega \times \mathbb{S}^{n-1})$ and vanishes near $\partial \Omega$.

We will call any unit speed curve $\gamma_{x,\theta}$ a \textit{regular broken ray} in $\Omega$ if
\begin{enumerate}
\item[(a)] its associated tangent vector field $(\gamma_{x,\theta}(t), \dot{\gamma}_{x,\theta}(t))$ intersects $\Gamma_{-}(E)$\\
\item[(b)] it consists of finitely many straight line segments $\gamma_{x,\theta,1}, \gamma_{x,\theta,2}, \ldots, \gamma_{x,\theta,N}$ before hitting $E$ again\\
\item[(c)] it does not intersect $\partial E$, and\\
\item[(d)] it obeys the geometrical optics reflection law 
\begin{equation}
\dot{\gamma}_{x,\theta,j+1}(0) = \dot{\gamma}_{x,\theta,j}(L_{j}) - 2\left\langle \nu(\gamma_{x,\theta,j}(L_{j})),\dot{\gamma}_{x,\theta,j}(L_{j}) \right\rangle \nu(\gamma_{x,\theta,j}(L_{j})), \,1 \leq j \leq N.
\end{equation}
\end{enumerate}
Here $L_{j}$ denotes the length of $\gamma_{x,\theta,j}$ and the subscripts refer to the fact that $(\gamma_{x,\theta}(0), \dot{\gamma}_{x,\theta}(0)) = (x,\theta)$.

We use the notation $T:\Gamma_{-} \to \Gamma_{-}$ to denote the billiard map mapping a vector $(x,\theta) \in \Gamma_{-}$ to $(x',\theta') \in \Gamma_{-}$, where $x'$ is the intersection point of the geodesic $\gamma_{x,\theta}$ with $\partial \Omega$, and $\theta'$ is the reflected direction. We use the maps $\pi_{1},\pi_{2}$ to denote the projections $\Gamma_{-} \to \partial \Omega$ and $\Gamma_{-} \to \mathbb{S}^{n-1}$, respectively. Also $\tau_{\pm}(x,\theta) = \min \{ t \geq 0 \, | \, x + t\theta \in \partial \Omega\}$ are the positive and negative travel times from a given unit tangent vector $(x,\theta) \in S\Omega$, and $N(x,\theta)$ for $(x,\theta) \in \Gamma_{-}(E)$ will denote the number of reflections of the given trajectory before intersecting $E$. We use the notation $g^{\#}(x,\theta) := g(x + \tau_{-}(x,\theta)\theta, \theta)$ to denote the extension of a function $g$ on $\Gamma_{-}$ to $\Omega \times \mathbb{S}^{n-1}$. 

\subsection{Unattenuated Case}
First we consider the unattenuated broken ray transform $I_{0,E}$ with respect to the accessible set $E$, so as not to be distracted by technicalities. The operator $I_{0,E}:L^{2}(\Omega) \to L^{2}(\Gamma_{-}, d\Sigma)$, where $d\Sigma = |\nu(x) \cdot \theta|\, dS(x)\, d\theta$ is the standard measure on $\Gamma_{-}$, is defined by
\begin{equation}
I_{0, E}f(\gamma_{x,\theta}) = I_{0, E}f(x,\theta) := \sum_{j=0}^{N(x,\theta)}\int_{\rn{+}} f(\pi_{1} \circ T^{j}(x,\theta) + t \pi_{2}\circ T^{j}(x,\theta))\,dt. \label{forwardoperator}
\end{equation}
for all regular broken rays $\gamma_{x,\theta}$. Here 
\begin{equation}
(z_{j}(x,\theta),\theta_{j}(x,\theta)) = T^{j}(x+\tau_{-}(x,\theta)\theta,\theta), \quad (x,\theta) \in S\Omega. \label{eq:zjthetaj}
\end{equation}

Since in general there may be regular broken rays in $\Omega$ with arbitrarily many reflections, we cannot expect for $I_{0,E}:L^{2}(\Omega) \to L^{2}(\Gamma_{-}, d\Sigma)$ to be a bounded operator. Thus, we will instead consider the cutoff operator $\alpha I_{0, E}$, where $\alpha:\Gamma_{-}\to \mathbb{R}$ is a smooth cutoff function supported on some subset of broken rays that have at most $N_{max} \geq 0$ reflections. The boundedness of this cutoff operator, with possible attenuation, will be proved later when we introduce the more general weighted transform.

\begin{remark*}Although general polyhedral domains lack obvious reflection rules at the corners (some authors ignore such trajectories and some define them in terms of limits of regular trajectories, e.g. \cite{eskin}), the special symmetrical nature of the square does give an unambiguous rule for such reflections. But one still runs into the problem that the composite functions $z_{j}(x,\theta), \theta_{j}(x,\theta)$ are not real analytic for $(x,\theta)$ associated with a corner point. One simple, albeit more restrictive route, is to assume that the subset $E$ also includes a small neighborhood of each corner. In that case, all broken rays with corner points are just ballistic X-rays. More generally, one can cut out the rays with corner points from the measurement operator. However, if one assumes the attenuation coefficient to vanish in a neighborhood of the boundary, it is still possible to treat rays with corners as we will discuss later.\end{remark*}

It is easy to see that $N(x,\theta)$ is piecewise constant, which leads to the unfortunate fact that $I_{0, E}$ in general does not preserve the smoothness of a function. Since we are considering a convex domain (flat) in this work, the discontinuities of $N(x,\theta)$ occur only for $(x,\theta) \in \Gamma_{-}(E)$ such that $\gamma_{x,\theta}$ has a reflection occurring at some point in $\partial E$. This is the reason they are not classified as \textit{regular} broken rays. To eliminate such discontinuities introduced by $N(x,\theta)$ we will use a collection of smooth cutoff functions $\alpha_{k}$ on $\Gamma_{-}(E)$ so as to remove the bad parts of $I_{0, E}$. For example, for each point $(x_{0},\theta_{0}) \in \Gamma_{-}(E)$ where $N(x,\theta)$ is continuous (i.e. $\gamma_{x_{0},\theta_{0}}$ and nearby rays do not touch corner points), we may multiply $I_{\sigma, E}$ by a smooth function $\alpha(x,\theta) \equiv 1$ near $(x_{0}, \theta_{0})$ and such that $N(x,\theta) \vert_{\mathrm{supp}(\alpha)}$ is constant. This localized operator will have a constant number of reflections and thus behaves very much like the usual generalized ray transform, which makes it possible to apply the microlocal analytic techniques of \cite{xraygeneric}. Of course, additional care is needed to deal with broken rays that contain corner points, which we discuss later. The central idea is that for fixed $(x,\xi) \in T^{*}\Omega \setminus 0$, any broken ray $\gamma$ passing through $x$ with direction normal to $\xi$ will allow for the detection of $(x,\xi)$ in $\mathrm{WF}_{A}(f)$. From hereon we use the notation $\mathrm{WF}_{A}(f)$ to denote the analytic wavefront set of $f$.

The main idea we use in order to apply microlocal analytic techniques to the operator is a well-known method of billiards called path unfolding. Essentially, we can unfold any broken ray in $\Omega$ into a straight line on a different domain $\widetilde{\Omega}$ defined by reflecting $\Omega$ across each of its edges repeatedly to obtain a tiling of $\rn{2}$. Similarly, we can extend $f$ to a function $\widetilde{f}$ on $\widetilde{\Omega}$ by reflection across $\partial \Omega$. We will use indices $(l_{1},l_{2}) \in \mathbb{Z}^{2}$ to denote how many horizontal and vertical reflections a given broken ray undergoes on the unfolded domain $\widetilde{\Omega}$ before returning to the set $E$. We call this the \textit{reflection signature} of a path. For example, $(-1,2)$ indicates $1$ reflection to the left and $2$ reflections upwards. The following transformations and their respective inverses will be useful when translating between vectors on
$T\Omega$ and $T \widetilde{\Omega}$. Given a reflection signature $(l_{1},l_{2})$, we can compute the reflected version of a point $w \in \Omega$ by
\begin{align}
\widetilde{w} = R_{l_{1},l_{2}}(w) & = \left( l_{1} + \frac{1 - (-1)^{l_{1}}}{2} + (-1)^{l_{1}}w_{1}, l_{2} + \frac{1 - (-1)^{l_{2}}}{2} + (-1)^{l_{2}}w_{2} \right)\\
w = R_{l_{1},l_{2}}^{-1}(\widetilde{w}) & = \left( (-1)^{l_{1}+1}l_{1} + \frac{1-(-1)^{l_{1}}}{2} + (-1)^{l_{1}}\widetilde{w}_{1}, \, (-1)^{l_{2}+1}l_{2} + \frac{1-(-1)^{l_{2}}}{2} + (-1)^{l_{2}}\widetilde{w}_{2} \right).
\end{align}
Similarly, given an angle $\theta$, $S_{l_{1},l_{2}}(\theta)$ is the angle after an $(l_{1},l_{2})$ reflection. 
\begin{align}
\eta = S_{l_{1},l_{2}}(\theta) &= (-1)^{l_{1}+l_{2}}\theta + \frac{\pi}{2}\left[ 1 - (-1)^{l_{1}} \right]\\
\theta = S_{l_{1},l_{2}}^{-1}(\eta) & = (-1)^{l_{2}}\left( (-1)^{l_{1}} \eta + \frac{1-(-1)^{l_{1}}}{2}\pi \right).
\end{align}
Such formulas can be verified with straightforward geometry. Using this notation, we have
\begin{equation}
\widetilde{f}(x) = \left\{ \begin{array}{ll}
f(R_{l_{1},l_{2}}^{-1}(x)) & \textrm{ if }x \in [l_{1},l_{1}+1] \times [l_{2},l_{2}+1]\\
& (l_{1},l_{2}) \in \mathbb{Z}^{2},
\end{array}\right.
\end{equation}
which obviously makes sense even for compactly supported distributions in $\Omega$.

\begin{figure}
\centering
\def \svgwidth{0.7\columnwidth}
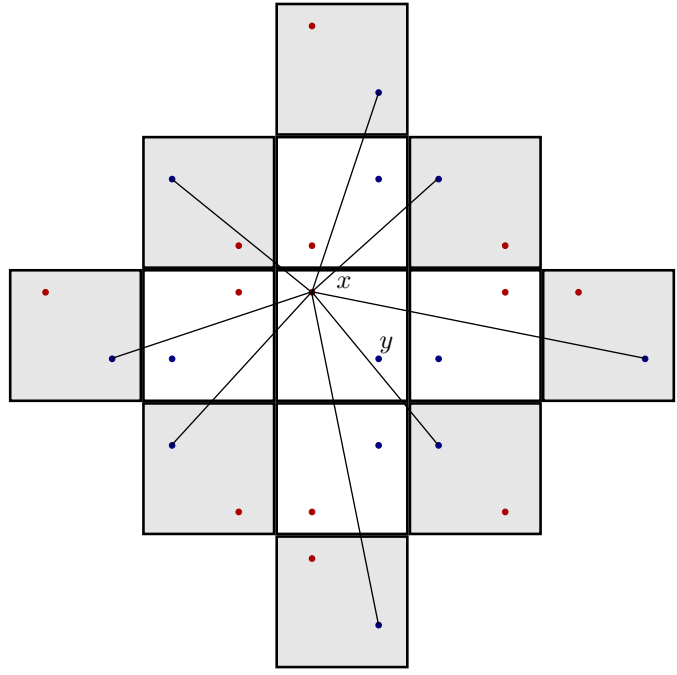
\caption{$2$-reflected unfolded paths from $x$ to $y$.\label{unfoldedpaths}}
\end{figure}

Note that one can also write $I_{0,E}f$ in the form
\begin{equation}
I_{0, E}f(x,\theta) = \int_{0}^{L(\gamma_{x,\theta})}\widetilde{f}(x+t\theta)\,dt,
\end{equation}
which resembles the usual X-ray transform. This allows for the treatment of compactly supported distributions $f \in \mathcal{E}'(\Omega)$ in Proposition \ref{thm:microlocalsquare} similar to in the analogous result of \cite{xraygeneric}. In particular, let $\alpha(x,\theta)$ be a smooth cutoff function on $\Gamma_{-}$ that localizes around some neighborhood of broken rays with the same reflection signature. Because $\widetilde{f}$ vanishes outside of the extended domain $\widetilde{\Omega}$, we may assume all broken rays in the support of $\alpha$ have the same length $L$, by extending them outside if necessary. Note that such rays may contain a corner point. We can then make sense of $\alpha I_{0,E}f$ where $f$ is a distribution. Specifically, if $f \in C_{0}^{\infty}(\Omega)$ and $\phi \in C_{0}^{\infty}(\rn{2})$ we have
\begin{align*}
 \int_{\Gamma_{-}} \alpha(x,\theta)I_{0,E}f(x,\theta)\phi(x,\theta)\,d\Sigma & = \int_{\Gamma_{-}}\alpha(x,\theta) \left( \int_{0}^{L}\widetilde{f}(x + t\theta)\,dt \right) \phi(x,\theta)\,d\Sigma\\
 & = \int_{\rn{2} \times \mathbb{S}^{1}}\alpha^{\#}(z,\theta) \phi^{\#}(z,\theta)\widetilde{f}(z) \,dz\,d\theta.
\end{align*}
Thus,
\begin{equation}
\langle \alpha I_{0,E}f, \phi \rangle = \left\langle \widetilde{f}, \int_{\mathbb{S}^{1}}\alpha^{\#}(\cdot, \theta)\phi^{\#}(\cdot, \theta)\,d\theta \right \rangle,
\end{equation}
which is well defined for $f \in \mathcal{E}'(\Omega)$. Note that we only need $\widetilde{f}$ to be well-defined on the support of $\alpha^{\#}$. On the square this is inconsequential since $\widetilde{f}$ is always uniquely defined on all of $\rn{2}$. However, this could prove useful if dealing with more general domains in future work.

For a given $N_{max} \in \mathbb{N}$, we define the \textit{visible set} by
\begin{align}
\mathcal{M} & := \{x \in \Omega \, | \, \forall \xi \in \mathbb{S}^{1}, (x,\xi) \in N^{*}\gamma \textrm{ for some regular broken ray $\gamma$}\\
& \hspace{7cm} \textrm{ with $N(\gamma) \leq N_{max}$}\}. \notag
\end{align}
\begin{remark*}
One simple example of what $\mathcal{M}$ looks like is the case when $E$ contains two adjacent edges of the square $\Omega = (0,1)^2$. Then $\mathcal{M} = \Omega$.

Similarly, if $\Omega = \{x \in \rn{n} \, | \, |x| \leq 1 \}$, then for $E$ containing a hemisphere we have $\mathcal{M} = \Omega$.

In general, if $E = \bigcup_{j}E_{j}$ is a union of connected open subsets $E_{j}$ of the boundary, then we always have $\bigcup_{j} \mathrm{ch}(E_{j})\subset \mathcal{M}$, where $\mathrm{ch}$ denotes the closed convex hull. However, this is the part of $\mathcal{M}$ that is attributed to only ballistic (unreflected) X-rays.
\end{remark*}

We then have the following injectivity result for the 2-dimensional square with Euclidean broken rays:
\begin{theorem}Let $E \subset \partial \Omega$ be open, let $N_{max} \in \mathbb{N}$, let $\mathcal{M}$ be the corresponding visible set,
and let $K \Subset \mathcal{M}$. Then $I_{0, E}$ is injective on $L^{2}(K)$. \label{thm:injectivesquare}\end{theorem}

\subsection{Attenuated Case}
Now we would like to make sense of nonzero attenuation in the definition of the broken ray transform with respect to $E \subset \partial \Omega$. We can then get injectivity and stability for attenuations which are near $0$ in $C^{2}$ and which vanish in a neighborhood of $\partial \Omega$.

For the moment, consider an attenuation function $\sigma \in C_{0}^{\infty}(\Omega \times \mathbb{S}^{1})$ which vanishes near $\partial \Omega$. The broken ray transform with respect to $E\subset \partial \Omega$ and with attenuation $\sigma$, denoted by $I_{\sigma, E}:L^{2}(\Omega) \to L^{2}(\Gamma_{-}, d\Sigma)$, is defined as
\begin{align}
& \quad I_{\sigma, E}f(\gamma_{x,\theta}) = I_{\sigma, E}f(x,\theta) \notag \\
& := \sum_{j=0}^{N(x,\theta)}\int_{\rn{+}} \exp\left( - \sum_{m=0}^{j-1}\int_{\rn{+}}\sigma(\pi_{1} \circ T^{m}(x,\theta) + \tau \pi_{2}\circ T^{m}(x,\theta), \pi_{2} \circ T^{m}(x,\theta))\,d\tau\right) \notag\\
& \qquad \cdot \exp\left(-\int_{\rn{+}}\sigma(\pi_{1}\circ T^{j}(x,\theta) + (t-\tau)\pi_{2}\circ T^{j}(x,\theta), \pi_{2} \circ T^{j}(x,\theta) )\,d\tau\right) \notag \\
& \qquad \cdot f(\pi_{1} \circ T^{j}(x,\theta) + t \pi_{2}\circ T^{j}(x,\theta))\,dt. \label{forwardoperator}\\
& = \sum_{j=0}^{N(x,\theta)}\int_{\rn{+}} \left[w_{j}f\right](\pi_{1} \circ T^{j}(x,\theta) + t \pi_{2}\circ T^{j}(x,\theta), \pi_{2}\circ T^{j}(x,\theta))\,dt. \notag\\
& = \sum_{j=0}^{N(x,\theta)}\int_{\rn{+}} \left[w_{j}f\right](z_{j}(x,\theta) + t \theta_{j}(x,\theta), \theta_{j}(x,\theta))\,dt. \notag
\end{align}
for all regular broken rays $\gamma_{x,\theta}$. The functions $z_{j}, \theta_{j}$ are defined in (\ref{eq:zjthetaj}). The weight functions $w_{j}$ on $S\Omega$ are given by
\begin{align}
w_{j}(y,\eta) & = \exp\left( -\sum_{m=0}^{j-1} \int_{\mathbb{R}_{+}}\sigma( z_{m-j}+\tau \theta_{m-j},\theta_{m-j})\,d\tau\right) \exp\left( -\int_{\mathbb{R}_{+}}\sigma(y - \tau \eta, \eta)\,d\tau \right)\notag\\
& = \exp\left( -\sum_{m=1}^{j} \int_{\mathbb{R}_{+}}\sigma( z_{-m}+\tau \theta_{-m},\theta_{-m})\,d\tau\right) \exp\left( -\int_{\mathbb{R}_{+}}\sigma(y - \tau \eta, \eta)\,d\tau \right)\notag
\end{align}
It is important to note that because $\sigma$ vanishes near $\partial \Omega$, the weight functions $w_{j}$ are smooth. The reason is that we can circumscribe $\Omega$ with the ball $B = B\left( \left( \frac{1}{2}, \frac{1}{2} \right), \frac{\sqrt{2}}{2} \right)$ and then replace $z_{m-j}(y,\eta) = \pi_{1} \circ T^{m-j}(y+ \tau_{-}(y,\eta)\eta,\eta)$ with the last point of $\partial B$ that intersects the line $z_{m-j}(y,\eta) + t\theta_{m-j}(y,\eta)$ for $t \leq 0$. This point is smoothly dependent on $(y,\eta) \in \Omega \times \mathbb{S}^{n-1}$. 

Just as we could rewrite the unattenuated broken ray transform as a standard X-ray transform for the augmented function $\widetilde{f}$, we can do the same for the attenuated case. First we define an extension of $\sigma$ on $\widetilde{\Omega}\times \mathbb{S}^{1}$ by
\begin{equation}
\widetilde{\sigma}(x,\theta) = \left\{ \begin{array}{ll}
\sigma( R_{l_{1},l_{2}}^{-1}(x), S_{l_{1},l_{2}}^{-1}(\theta)) & \textrm{ for }x \in [l_{1},l_{1}+1] \times [l_{2},l_{2}+1],\\
&  \quad (l_{1},l_{2}) \in \mathbb{Z}^{2}.
\end{array}\right.
\end{equation}
Then $\widetilde{w}(x,\theta) := \exp\left( -\int_{0}^{L(\gamma_{x,-\theta})}\widetilde{\sigma}(x-\tau\theta,\theta)\,d\tau\right)$ is the associated weight for $(x,\theta) \in \widetilde{\Omega} \times \mathbb{S}^{1}$. We've used the notation $L(\gamma_{x,-\theta})$ to mean the length of the half broken ray with initial condition $(x,-\theta)$ before hitting some copy of $E$ in $\widetilde{\Omega}$. Obviously, $\sigma \equiv 0$ implies that $w \equiv 1$, which is analytic. However, for general $\sigma$ that are analytic on $\Omega$, it is not generally true that $\widetilde{\sigma}$ is analytic. In fact, it is easy to see that $\widetilde{\sigma}$ has symmetries somewhat like that of a doubly periodic function, which if analytic would necessarily be constant (the Weierstrass $\wp$-function is doubly periodic but has poles). This is the reason why we start with the case that the weight function is identically $1$ (equivalently $\sigma = 0$) and then prove stability for $C^{2}$ perturbations of $\sigma$ that remain $0$ in a neighborhood of $\partial \Omega$.

For the purposes of this work, we need $\alpha I_{\sigma, E}: L^{2}(\Omega, \mathbb{S}^{n-1}) \to L^{2}(\Gamma_{-}, d\Sigma)$ to be bounded, where $\alpha$ is a smooth cutoff function which is supported only for broken rays that have at most $N_{max} \geq 0$ reflections. 
\begin{lemma}Let $\alpha \in C_{0}^{\infty}(\Gamma_{-})$ and suppose all regular broken rays in $\mathrm{supp}(\alpha)$ have $\leq N_{max} \in \mathbb{N}$ reflections. Then $I_{\sigma, E}:L^{2}(\Omega, \mathbb{S}^{n-1}) \to L^{2}(\Gamma_{-}, d\Sigma)$ is bounded.\end{lemma}
\begin{proof}
First we recall an identity from \cite{inversesource} which asserts that for any function $f \in L^{2}(\Omega, \mathbb{S}^{n-1})$, we have
\begin{equation*}
\int_{\Gamma_{-}}\int_{\mathbb{R}_{+}} f(x+t\theta, \theta)\,dt\,d\Sigma = \int_{\Omega \times \mathbb{S}^{n-1}} f(x,\theta)\,dx\,d\theta.
\end{equation*}
Now observe that
\begin{align*}
\| \alpha I_{\sigma, E} f(x,\theta) \|_{L^{2}(\Gamma_{-})}^{2} & \leq 2 \int_{\Gamma_{-}} |\alpha(x,\theta)|^{2}\sum_{j=0}^{N(x,\theta)}\left| \int_{\mathbb{R}_{+}} [w_{j} f](z_{j}(x,\theta) + t\theta_{j}(x,\theta), \theta_{j}(x,\theta))\,dt \right|^{2} \, d\Sigma\\
&\leq 2 \int_{\Gamma_{-}} \sum_{j=0}^{N_{max}}\left| \int_{\mathbb{R}_{+}} \chi_{[0,\tau_{+}(z_{j},\theta_{j})]}(t)[w_{j} f](z_{j}(x,\theta) + t\theta_{j}(x,\theta), \theta_{j}(x,\theta))\,dt \right|^{2} \, d\Sigma\\
& \leq 2 \int_{\Gamma_{-}} \sum_{j=0}^{N_{max}} \|\chi_{[0,\tau_{+}(z_{j},\theta_{j})]}\|_{L^{2}(\mathbb{R}_{+})}^{2}\|[w_{j}f](z_{j}+t\theta_{j},\theta_{j})\|_{L^{2}(\mathbb{R}_{+})}^{2}\,d\Sigma\\
& \leq 2 \int_{\Gamma_{-}} \sum_{j=0}^{N_{max}} \mathrm{diam}(\Omega)\|f(z_{j}+t\theta_{j},\theta_{j})\|_{L^{2}(\mathbb{R}_{+})}^{2}\,d\Sigma\\
& = 2 \int_{\Gamma_{-}} \sum_{j=0}^{N_{max}} \mathrm{diam}(\Omega)\int_{\mathbb{R}_{+}}|f(z_{j}+t\theta_{j},\theta_{j})|^{2}\,dt\,d\Sigma\\
& = 2 (N_{max}+1)\mathrm{diam}(\Omega)\int_{\Gamma_{-}} \int_{\mathbb{R}_{+}}|f(x+t\theta,\theta)|^{2}\,dt\,d\Sigma\\
& = 2 (N_{max}+1)\mathrm{diam}(\Omega)\|f\|_{L^{2}(\Omega \times \mathbb{S}^{n-1})}^{2}.
\end{align*}
\end{proof}

To formulate a stability estimate, we must first parametrize a family of regular broken rays, having an upper bound on the number of reflections, and whose conormal bundles cover $T^{*}\Omega$. Fix a positive $N_{max} \in \mathbb{N}$. For each regular broken ray $\gamma_{x_{0},\theta_{0}}$ with $N(x_{0},\theta_{0}) \leq N_{max}$, let $(l_{1},l_{2})$ be it's reflection signature on $\widetilde{\Omega}$. If $\gamma_{x_{0},\theta_{0}}$ has no corner points, then find a maximal connected neighborhood $U \subset \Gamma_{-}(E)$ containing $(x_{0},\theta_{0})$ such that all broken rays with initial data in $U$ are regular, have no corner points, and have the same reflection signature $(l_{1},l_{2})$ (See Figure \ref{fig:beam}). In particular, $U$ itself contains no corner points. We then choose a smooth cutoff function $\alpha_{k}$ on $\Gamma_{-}$ that is equal to $1$ on a compact subset of $U$ and vanishes outside of $U$. 

\begin{figure}
\centering
\def \svgwidth{0.6\columnwidth}
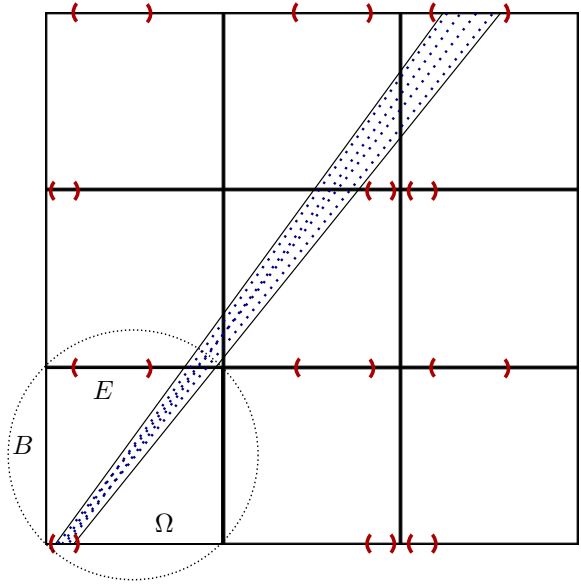
\caption{A typical beam of unfolded broken rays which follow the same reflection sequence before hitting $E$. The beam terminates in the part of $\widetilde{\Omega}$ with reflection signature $(h_{4},v_{4}) = (2,2)$.\label{fig:beam}}
\end{figure}

If $x_{0}$ is a corner point, then we consider all directed lines $\mathcal{L} = \{(x+t\theta,\theta) \, | \, t \in \mathbb{R}\}$ for $(x,\theta)$ in some neighborhood of $(x_{0},\theta_{0})$ such that $N(x,\theta)$ is constant. Then take the intersection of $\mathcal{L}$ with the ball $B = \left( \left( \frac{1}{2}, \frac{1}{2}\right), \frac{\sqrt{2}}{2} \right)$ of radius $\frac{\sqrt{2}}{2}$ whose boundary $\partial B$ circumscribes $\Omega$. We then define a smooth function $\alpha$ on $\partial B \times \mathbb{S}^{1}$ supported in $\partial B \times \mathbb{S}^{1} \cap \mathcal{L}$. We may then extend $\alpha$ to $\rn{2} \times \mathbb{S}^{1}$ to be constant on all directed lines $(x+t\theta, \theta)$. Its restriction to $\Gamma_{-}(\partial \Omega)$ is the cutoff we seek. We will treat the case that $\gamma_{x_{0},\theta_{0}}$ contains an additional corner point in \S \ref{sec:corners}. 

Let $\alpha = \sum_{k=1}^{N_{\alpha}}\alpha_{k}$ be a sum of smooth cutoff functions on $\Gamma_{-}(E)$ in the sense of the previous paragraph, and consider the regularized operator
\begin{equation}
I_{\sigma,E,\alpha} := \sum_{k=1}^{N_{\alpha}}\alpha_{k} I_{\sigma,E}.
\end{equation}
Then
\begin{equation}
N_{\sigma, E, \alpha} := I_{\sigma, E,\alpha}^{*} I_{\sigma, E, \alpha} =  \sum_{k_{1}=1,k_{2}=1}^{N_{\alpha}}I_{\sigma,E}^{*}\alpha_{k_{1}}\alpha_{k_{2}}I_{\sigma, E}.
\end{equation}
The transpose is taken with respect to the standard measure $|\nu(x) \cdot \theta|\,dS(x)\,d\theta$ on $\Gamma_{-}(E)$. Note that we may as well assume that the $\alpha_{k}$ functions have disjoint support, since it won't change the kind of terms present in the normal operator (if two such cutoff functions $\alpha_{k_{1}}$ and $\alpha_{k_{2}}$ have overlapping support, then their product is also smooth with the supported broken rays localized just as in the previous construction).

We also define the \textit{microlocally visible set} by
\begin{equation}
\mathcal{M}' := \{ (z,\xi) \in T^{*}\Omega \, | \, (z,\xi) \in N^{*}\gamma_{x,\theta} \textrm{ for some }(x,\theta) \textrm{ with }\alpha(x,\theta) > 0\}.
\end{equation}

\begin{remark*}One reason we need to impose an upper bound on the number of reflections is that as $N$ increases, the open level sets of $N(x,\theta)$ shrink in measure to zero. Similarly, the neighborhoods in $\Gamma_{-}$ on which we have constant reflection signature $(l_{1},l_{2})$ become infinitesimal as $|l_{1}|+|l_{2}| \to \infty$. Thus $\alpha$ will not be smooth unless we apply some threshold.\end{remark*}

We have the following stability result analogous to that in \cite{xraygeneric}.
\begin{theorem}
\begin{enumerate}
\item[(a)] Fix $\sigma$ that is $0$ in a neighborhood of $\partial \Omega$, and choose $\alpha \in C^{\infty}(\Gamma_{-})$ as above with reflection threshold $N_{max}$. Fix a set $K \Subset \mathcal{M}$ compactly contained in the visible set $\mathcal{M}$. If $I_{\sigma, E, \alpha}$ is injective on $L^{2}(K)$, then
\begin{equation}
\frac{1}{C}\|f\|_{L^{2}(K)} \leq \|N_{\sigma, E, \alpha}f \|_{H^{1}(\Omega)} \leq C\|f\|_{L^{2}(K)}. \label{eq:stabilityestimate}
\end{equation}
\item[(b)] Let $\alpha^{0}$ be as above related to some fixed $\sigma_{0}$. Assume that $I_{\sigma_{0}, E, \alpha^{0}}$ is injective on $L^{2}(K)$. Then estimate (\ref{eq:stabilityestimate}) remains true for $(\sigma, \alpha)$ in a small $C^{2}$ neighborhood of $(\sigma_{0}, \alpha^{0})$ such that $\sigma \equiv 0$ near $\partial \Omega$, and with a uniform constant $C > 0$.
\end{enumerate}\label{thm:stabilityestimate}
\end{theorem}

\section{Injectivity of $I_{0, E}$ for the Square\label{sec:squareinjective}}
For the 2-dimensional square, we will be able to establish injectivity of the broken ray transform for the unknown function $f$ supported on some subset of the domain that depends on $E$. We will be able to detect covectors from $N^{*}\gamma_{x,\theta}$ in $\mathrm{WF}_{A}(f)$ for (most) $(x,\theta)\in \Gamma_{-}(E)$. The idea for the proof of the theorem is heavily borrowed from the microanalytic techniques utilized in \cite{xraygeneric}. We will need to assume that $f$ is compactly supported within $\Omega$, which in particular will allow us to treat the reflected ray as a smooth regular curve in the sense of \cite{xraygeneric} since the integrand vanishes near the points of reflection. By using the path unfolding approach, we can relatively easily obtain the following microlocal result.

\begin{proposition}Let $f \in \mathcal{E}'(\Omega)$. Suppose $I_{0,E}f(x,\theta) = 0$ for all $(x,\theta)$ in a small neighborhood of $(x_{0},\theta_{0}) \in \Gamma_{-}(E)$ where $\gamma_{x_{0},\theta_{0}}$ is a regular broken ray. Then $\mathrm{WF}_{A}(f) \cap N^{*}\gamma_{x_{0},\theta_{0}} = \emptyset$.\label{thm:microlocalsquare}\end{proposition}

\begin{proof}We need to construct coordinates near the broken ray $\gamma_{0}$. To do this, we let $\{x_{0,1},\ldots, x_{0,N}\}$ be the ordered reflection points on $\partial \Omega$ for $\gamma_{0}$, where $N = N(x_{0}, \theta_{0})$. We then construct a new domain $\widetilde{\Omega}$ consisting of $N$ reflected copies of $\Omega$ which are glued together with the original $\Omega$ along the edges that contain each $x_{0,j}$, $1 \leq j \leq N$. Specifically, let $P_{m}= \{x \in \rn{2} \, | \, (x-x_{0,j}) \cdot \xi_{j} = 0\}$ for $1 \leq j \leq N$ be the hyperplanes in $\rn{2}$ corresponding to the reflection points of $\gamma_{0}$. We iteratively construct $\widetilde{\Omega}$ by starting with $x_{0,1}$ and reflecting $\Omega$ across the hyperplane $P_{1}$. Then glue the two copies of $\Omega$ together along $P_{1}$. Also reflect each hyperplane $P_{j}$ across $P_{1}$ for $2 \leq j \leq N$. Now we consider the reflected hyperplane $P_{2}$, which we still denote with the same notation, and reflect the previously reflected copy of $\Omega$ across it. Continuing in this way, we obtain a space consisting of $N+1$ copies of $\Omega$, each reflected along a specific edge. If $\gamma_{0}$ contains a corner point, we can add additional reflected copies of $\Omega$ as necessary to accommodate nearby broken rays. We also obtain a distribution $\widetilde{f} \in \mathcal{E}'(\widetilde{\Omega})$ and an unfolded version of $\gamma_{0}$, denoted by $\widetilde{\gamma_{0}}$, using the same process. Note that by reversing the reflection sequence, any covector $(\widetilde{z}, \widetilde{\xi}) \in T^{*}\widetilde{\Omega}$ corresponds to a unique covector $(z,\xi) \in T^{*}\Omega$.

Since the edges of $\Omega$ are flat, it follows that $\widetilde{\gamma_{0}}$ is a straight line. The openness of $E$ and the assumptions on $(x_{0},\theta_{0})$ ensure that for $(x,\theta)$ in a neighborhood $V \subset \Gamma_{-}(E)$ of $(x_{0}, \theta_{0})$, the unfolded paths $\widetilde{\gamma}_{x,\theta}$ lie in the same extended space $\widetilde{\Omega}$ and avoid $\partial E$. Now choose a point $p_{0}$ on the line $x_{0} + t\theta_{0}$ for $t < 0$ so that $p_{0} \notin \widetilde{\Omega}$. Define $x = p_{0} + t\theta$. Then $(\theta, t)$ are local coordinates near any point of $\widetilde{\Omega} \cap \widetilde{\gamma}_{x_{0},\theta_{0}}$. We can assume without loss of generality that $\theta_{0} = (0,1)$. Writing $x = (x^{1},x^{2}) = (\theta^{1}, t)$, we have the coordinates
\begin{equation*}
U = \{x = (\theta^{1}, t) \, | \, |\theta^{1}| < \epsilon, \, l^{-} < t < l^{+} \} \subset \widetilde{\Omega}.
\end{equation*}
Since $f$ is compactly supported inside of $\Omega$, we may shrink $\epsilon$ as necessary so that we can take $l^{-}, l^{+}$ to be constant. Now as in the proof of Proposition 1 in \cite{xraygeneric}, we let $(z_{0}, \xi_{0} )\in N^{*}\gamma_{0}$. Then consider the corresponding reflected covector via path unfolding given by $(\widetilde{z}^{0}, \widetilde{\xi}_{0}) \in N^{*}\widetilde{\gamma}_{0}$. That is, let $\{\xi_{1},\xi_{2},\ldots, \xi_{m}\}$ the sequence of normal vectors corresponding to the ordered reflection points of the path $\gamma_{0}$ up until reaching the point $z_{0}$. Let $R_{j}:\rn{2} \to \rn{2}$ be the reflection operators across each affine hyperplane $z \cdot \xi_{j} = x_{0,j} \cdot \xi_{j}$ for $1 \leq j \leq m$. Specifically, $R_{j}v = v- 2((v-x_{0,j})\cdot \xi_{j})\xi_{j}$. We then define $\widetilde{z}_{0} = Rz_{0} := R_{1}R_{2} \cdots R_{m}z_{0}$. Also, consider the operators $A_{j}:\rn{2} \to \rn{2}$ defined as reflections across the hyperplanes $z \cdot \xi_{j} = 0$ that pass through the origin. Specifically, $A_{j}v = v - 2(v \cdot \xi_{j})\xi_{j}$. We then define $\widetilde{\xi}_{0} = A\xi_{0} := A_{1}A_{2} \cdots A_{m_{z_{0}}}\xi_{0}$. The argument of Proposition 1 in \cite{xraygeneric} shows that $(\widetilde{z}_{0}, \widetilde{\xi}_{0}) \notin \mathrm{WF}_{A}(\widetilde{f})$. 

Now we undo the reflection process to conclude that $(z_{0}, \xi_{0}) \notin \mathrm{WF}_{A}(f)$. In particular, we use the respective inverse transformations $A^{-1}$ and $R^{-1}$, generated by $A_{j}^{-1} = A_{j}$ and $R_{j}^{-1} = R_{j}$ to recover $(z_{0},\xi_{0})$ from $(\widetilde{z}_{0}, \widetilde{\xi}_{0})$. Note that the reflections $A_{j}, R_{j}$ preserve the singularities of $f$, and so the analytic wavefront set of $\widetilde{f}$ transforms in the obvious way. Thus $(\widetilde{z}_{0}, \widetilde{\xi}_{0}) = (Rz_{0}, A\xi_{0}) \notin \mathrm{WF}_{A}(\widetilde{f}) \Longrightarrow (z_{0}, \xi_{0}) \notin \mathrm{WF}_{A}(f)$.\end{proof}

\begin{proof}[Proof of Theorem \ref{thm:injectivesquare}] Suppose $I_{\sigma, E}f = 0$. By Proposition \ref{thm:microlocalsquare} we have that $f$ is analytic on $K$. Since $\mathrm{supp}(f) \subset K$, we have that $f$ can be extended to an entire function and hence must be identically zero.\end{proof}

\begin{remark*}In the case of the square Theorem \ref{thm:injectivesquare} gives injectivity on the whole domain if $E$ contains two adjacent (closed) edges. However, one must be careful since it fails for $E$ consisting of two opposite edges. \end{remark*}

\section{Stability \label{sec:stability}}
In order to prove the stability result for $I_{\sigma,E}$ with $\sigma$ a $C^{2}$ perturbation of $0$ that vanishes near $\partial \Omega$, we will first compute its normal operator. Suppose there is an upper limit $N_{max}$ on the number of reflections we want to allow in the data, which is implicit in the choice of cutoff functions $\alpha_{k}(x,\theta)$, $1 \leq k \leq N_{\alpha}$. We construct the modified operator
\begin{equation}
I_{\sigma,E,\alpha}f(x,\theta) := \sum_{k=1}^{N_{\alpha}}\alpha_{k}(x,\theta)I_{\sigma,E}f(x,\theta), \quad (x,\theta) \in \Gamma_{-}(E).
\end{equation}
We will always suppose that $K \Subset \mathcal{M}$ is fixed, where $\mathcal{M} \subset \Omega$ is the visible set. We will prove the following proposition regarding the structure of $N_{\sigma,E,\alpha}$:
\begin{proposition}$N_{\sigma,E,\alpha} = N_{\sigma,E,\alpha,ballistic} + N_{\sigma,E,\alpha,reflect}$ where $N_{\sigma,E,\alpha,ballistic}$ is a classical pseudo differential operator of order $-1$, elliptic on $\mathcal{M}$, and $N_{\sigma, E,\alpha,reflect}$ is an operator with $C_{0}^{\infty}(\Omega \times \Omega)$ Schwartz kernel. Thus there exists a classical pseudo differential operator $Q$ in $\Omega$ of order $1$ such that
\begin{equation}
QN_{\sigma, E,\alpha}f = f + QN_{\sigma,E,\alpha,reflect}f + S_{1}
\end{equation}
for any $f \in \mathcal{D}'(K)$, where $S_{1}$ is an operator with $C_{0}^{\infty}(\Omega \times \Omega)$ Schwartz kernel. \label{prop:normaldecomposition}\end{proposition}

\subsection{Computing the Normal Operator\label{sec:normaloperatorcompute}}
Recall that given $j \in \mathbb{Z}$ we have defined the variables $z_{j},\theta_{j}$ depending on $x \in \Omega$, $\theta \in \mathbb{S}^{1}$ by
\begin{equation*}
(z_{j}, \theta_{j}) := T^{j}(x+\tau_{-}(x,\theta)\theta,\, \theta)
\end{equation*}
Notice that $(z_{0},\theta_{0}) = (x + \tau_{-}(x,\theta)\theta, \theta)$. For now, we will only assume that the domain $\Omega$ is an open subset of $\rn{n}$ with piecewise smooth, convex boundary. In \S \ref{sec:normalsimplify} and beyond we then work specifically with the square $(0,1) \times (0,1)$.

Choose a cutoff function $\alpha(x,\theta)$ on $\Gamma_{-}$ such that $\alpha I_{\sigma, E}$ has a constant number of reflections $N \geq 0$ on its support. Later we will consider broken rays with corner points, but then the number of reflections is not constant, so it must be treated slightly differently. The adjoint to $\alpha I_{\sigma,E}$ can be computed as follows:
\begin{align}
& \quad \int_{\Gamma_{-}}\alpha(x,\theta) [I_{\sigma,E}f](x,\theta) g(x,\theta)\,d\Sigma \label{eq:adjointcomputation}\\
& = \int_{\Gamma_{-}}\alpha(x,\theta) g(x,\theta) \sum_{j=0}^{N}\int_{\rn{+}} w_{j}(z_{j}+t\theta_{j},\theta_{j})f(z_{j} + t \theta_{j})|\nu(x) \cdot \theta|\,dt\,dS(x)\,d\theta \notag
\end{align}

For each $j$ we make the change of variables $(x,\theta,t) \mapsto (y,\eta)$ where $(y,\eta) := (z_{j}(x,\theta)+t\theta_{j}(x,\theta), \theta_{j}(x,\theta))$. Thus
\begin{align}
(x,\theta) & = (z_{-j}(y,\eta), \theta_{-j}(y,\eta))\notag\\
t & = -\tau_{-}(y,\eta).
\end{align}
Intuitively, in moving from $(y,\eta)$ to $\left( z_{-j}(y,\eta), \theta_{-j}(y,\eta) \right)$, we first project from the point $y$ to the boundary along the direction $-\eta$ and then trace backwards along the trajectory for the previous $j$ reflections. We compute
\begin{align}
& \Big(z_{m}(x,\theta) + s\theta_{m}(x,\theta), \theta_{m}(x,\theta)\Big)\notag\\
& = \left(z_{m}\Big(z_{-j}(y,\eta), \theta_{-j}(y,\eta)\Big) + s\theta_{m}\Big(z_{-j}(y,\eta), \theta_{-j}(y,\eta)\Big), \theta_{m}\Big(z_{-j}(y,\eta), \theta_{-j}(y,\eta)\Big)\right)\notag\\
& = (z_{m-j}(y,\eta)+s\theta_{m-j}(y,\eta), \, \theta_{m-j}(y,\eta)).\notag
\end{align}

Given a fixed number of reflections $j$, it is clear that each point $(y, \eta) \in \Omega \times \mathbb{S}^{n-1}$ corresponds to a unique point $(x,\theta) \in \Gamma_{-}$. Thus, the previously described change of variables is a diffeomorphism. The resulting integral is over the domain $\Omega \times \mathbb{S}^{n-1}$.
We have
\begin{align*}
& \sum_{j=0}^{N}\int_{\Omega \times \mathbb{S}^{n-1}}\alpha(z_{-j}(y,\eta), \theta_{-j}(y,\eta))g(z_{-j}(y,\eta), \theta_{-j}(y,\eta))w_{j}(y,\eta)f(y) J_{j}(y,\eta)\,dy\,d\eta,
\end{align*}
where
\begin{equation*}
J_{j}(y,\eta) = \left| \frac{\partial x \partial \theta \partial t}{\partial y \, \partial \eta}\right| |\nu(z_{-j}(y,\eta))\cdot \theta_{-j}(y,\eta)|.
\end{equation*}
Since the billiard map $T$ preserves the symplectic form $|\nu(x) \cdot \theta| dx \wedge d\theta$ on $\Gamma_{-}$, it follows that $J_{j}(y,\eta) = 1$. We stress that this is even true for domains with piecewise smooth boundary as is the case for the square, by measure theoretic arguments (see \cite{serge2} \S 1.7). One can verify that the Jacobian factors above cancel by first making the change of variables $(x,\theta)\to T^{j}(x,\theta)$ in (\ref{eq:adjointcomputation}), using the invariance of the symplectic form $|\nu(x)\cdot \theta| dx\wedge d\theta$ under $T$. Then change to $(y,\eta)$ coordinates on $\Omega \times \mathbb{S}^{n-1}$ using the same approach as in (Theorem 1 (b), \cite{inversesource}). Thus we have
\begin{align}
(\alpha I_{\sigma, E})^{*}g(x) & = \sum_{j=0}^{N}\int_{\mathbb{S}^{n-1}}[\alpha g](z_{-j}(x,\theta), \theta_{-j}(x,\theta)) w_{j}(x,\theta) \,d\theta
\end{align}

We are now ready to compute the normal operator $I_{\sigma,E}^{*}I_{\sigma,E}$. By definition
\begin{align}
& I_{\sigma,E,\alpha}^{*}I_{\sigma,E,\alpha}f(x) \notag\\
& = \sum_{j_{1}=0}^{N}\int_{\mathbb{S}^{n-1}}|\alpha(z_{-j_{1}}(x,\theta), \theta_{-j_{1}}(x,\theta))|^{2}I_{\sigma,E}f(z_{-j_{1}}(x,\theta), \theta_{-j_{1}}(x,\theta))w_{j_{1}}(x,\theta)\, d\theta\notag\\
& = \sum_{j_{1}=0}^{N}\sum_{j_{2}=0}^{N} \int_{\mathbb{S}^{n-1}}\int_{\mathbb{R}_{+}}|\alpha(z_{-j_{1}}(x,\theta), \theta_{-j_{1}}(x,\theta))|^{2} \, w_{j_{1}}(x,\theta)\notag\\
& \quad \cdot w_{j_{2}}(z_{j_{2}-j_{1}}(x,\theta)+t\theta_{j_{2}-j_{1}}(x,\theta), \theta_{j_{2}-j_{1}}(x,\theta)) \, f\left(z_{j_{2}-j_{1}}(x,\theta) + t\theta_{j_{2}-j_{1}}(x,\theta)\right)\,dt\,d\theta \notag\\
& =\sum_{j_{1}=0}^{N}\sum_{j_{2}=0, j_{2} \neq j_{1}}^{N}\int_{\mathbb{S}^{n-1}}\int_{\mathbb{R}_{+}}w_{j_{1}}(x,\theta)\,w_{j_{2}}(z_{j_{2}-j_{1}}(x,\theta) + t\theta_{j_{2}-j_{1}}(x,\theta), \theta_{j_{2}-j_{1}}(x,\theta)) \notag\\
& \quad \cdot |\alpha(z_{-j_{1}}(x,\theta), \theta_{-j_{1}}(x,\theta))|^{2}\, f\left(z_{j_{2}-j_{1}}(x,\theta) + t\theta_{j_{2}-j_{1}}(x,\theta) \right)\,dt\,d\theta\notag\\
& + \sum_{j_{1}=0}^{N}\int_{\mathbb{S}^{n-1}}\int_{\mathbb{R}_{+}}\, w_{j_{1}}(x+t\theta,\theta)w_{j_{1}}(x,\theta) \, |\alpha(z_{-j_{1}}(x,\theta), \theta_{-j_{1}}(x,\theta))|^{2} \, f\left(x +\tau_{-}(x,\theta)\theta + t\theta \right)\,dt\,d\theta \notag\\
& =: N_{\sigma, E, \alpha, reflect}f(x) + N_{\sigma, E, \alpha, ballistic}f(x). \label{eq:normalop}
\end{align}
It follows that the full normal operator $N_{\sigma,E,\alpha}$, corresponding to the partition by smooth cutoff functions $\alpha_{k}$, is given by the sum
\begin{align*}
& N_{\sigma,E,\alpha}f(x) \\
& = \sum_{k=1}^{N_{\alpha}}\sum_{j_{1}=0}^{N_k}\sum_{j_{2}=0}^{N_k}\int_{\mathbb{S}^{n-1}}\int_{\mathbb{R}_{+}} |\alpha_{k}(z_{-j_{1}}, \theta_{-j_{1}})|^{2} w_{j_{1}}\left( x,\theta \right) w_{j_{2}}\left(z_{j_{2}-j_{1}}+ t\theta_{j_{2}-j_{1}}, \theta_{j_{2}-j_{1}}\right)\\
& \cdot f\left(z_{j_{2}-j_{1}}+ t\theta_{j_{2}-j_{1}}\right)\, dt\,d\theta
\end{align*}
We pick off the ballistic terms where $j_{1}=j_{2}$ to decompose $N_{\sigma,E,\alpha}f$ as a sum
\begin{equation*}
N_{\sigma,E,\alpha} = N_{\sigma,E,\alpha, ballistic} + N_{\sigma,E,\alpha,reflect}
\end{equation*}
with
\begin{align*}
& N_{\sigma,E,\alpha,ballistic}f(x)\\
& = \sum_{k=1}^{N_{\alpha}}\sum_{j=0}^{N_{k}}\int_{\mathbb{S}^{n-1}}\int_{\mathbb{R}_{+}} |\alpha_{k}(z_{-j}, \theta_{-j})|^{2} w_{j}\left( x,\theta \right) w_{j}\left(x + (\tau_{-}(x,\theta)+t)\theta, \theta\right)\\
& \qquad \cdot f\left(x + (\tau_{-}(x,\theta)+t)\theta \right)\, dt\,d\theta\\
& = \sum_{k=1}^{N_{\alpha}}\sum_{j=0}^{N_k}\int_{\mathbb{S}^{n-1}}\int_{\mathbb{R}} |\alpha_{k}(z_{-j}, \theta_{-j})|^{2} w_{j}\left( x,\theta \right) w_{j}\left(x + t\theta, \theta\right)f\left(x + t\theta \right)\, dt\,d\theta\\
& = \sum_{k=1}^{N_{\alpha}}\sum_{j=0}^{N_k}\int_{\mathbb{S}^{n-1}}\int_{\rn{+}} \left[ |\alpha_{k}(z_{-j}(x,\cdot) \theta_{-j}(x,\cdot))(\theta)|^{2} w_{j}\left( x,\cdot \right) w_{j}\left(x + t\theta, \cdot \right)\right]_{even}(\theta)\\
& \qquad \cdot f\left(x + t\theta \right)\, dt\,d\theta
\end{align*}
and
\begin{align*}
N_{\sigma,E,\alpha,reflect}f(x) & = \sum_{k=1}^{N_{\alpha}}\sum_{j_{1}=0}^{N_{k}}\sum_{j_{2}=0, j_{2} \neq j_{1}}^{N_k}\int_{\mathbb{S}^{n-1}}\int_{\mathbb{R}_{+}} |\alpha_{k}(z_{-j_{1}}, \theta_{-j_{1}})|^{2}\\
& \cdot w_{j_{1}}\left( x,\theta \right) w_{j_{2}}\left(z_{j_{2}-j_{1}}+ t\theta_{j_{2}-j_{1}}, \theta_{j_{2}-j_{1}}\right)f\left(z_{j_{2}-j_{1}}+ t\theta_{j_{2}-j_{1}}\right)\, dt\,d\theta.
\end{align*}
The notation $\left[ g(x,\cdot) \right]_{even}(\theta)$ is the even part of $g$ in the variable $\theta$, defined by $\left[ g(x,\cdot)\right]_{even}(\theta) = \frac{1}{2}\left[ g(x,\theta) + g(x,-\theta) \right]$. 

By Lemma 2 of \cite{xraygeneric} $N_{\sigma,E,\alpha, ballistic}$ is a classical pseudo differential operator of order $-1$ with principal symbol
\begin{equation}
a_{0}(x,\xi) = 2\pi \sum_{k=0}^{K}\sum_{j=0}^{N_{k}}\int_{\theta \in \mathbb{S}^{n-1}, \, \theta \cdot \xi = 0}|\alpha_{k}(z_{-j}, \theta_{-j})|^{2} \left| w_{j}\left( x,\theta \right)\right|^2 \,d\theta
\end{equation}
The bound on the number of reflections $N_{max}$ ensures that the integral kernel has a positive lower bound on the set of all covectors $(x,\xi)$ such that it is non vanishing at some $\theta \in \mathbb{S}^{n-1}$ normal to $\xi$. Clearly, $N_{\sigma,E,\alpha,ballistic}$ is also elliptic on the set $\mathcal{M}'$.

\subsection{Simplifying $N_{\sigma,E,\alpha,reflect}$ for the Square \label{sec:normalsimplify}}
For the reflected part of the normal operator, our goal will be to make a change of variables so that $f(z_{j_{2}-j_{1}} + t\theta_{j_{2}-j_{1}})$ becomes $f(y)$, with the hope that what we obtain is either a weakly singular integral operator (i.e. smoothing of order $-1$) or an operator with $C_{0}^{\infty}$ Schwartz kernel. It turns out that the latter will hold. This technique is very much inspired by the kind of substitution used in \cite{inversesource, hubenthal}, and also in \cite{xraygeneric} for the non-Euclidean case. For now, we will still be working with the part of the operator that avoids broken rays with corner points.

Let $(x,\theta) \in \Omega \times \mathbb{S}^{1}$, suppose we have some smooth cutoff $\alpha$ on $\Gamma_{-}(E)$. By construction, all broken rays originating from the support of $\alpha$ have the same sequence of reflection signatures. Let $\{(h_{1},v_{1}), \ldots, (h_{N},v_{N})\}$ be the ordered set of all such pairs (see Figure \ref{fig:beam}). Let $z = z_{j_{2}-j_{1}}(x,\theta)$ (i.e. $z$ is the $j_{2}^{\mathrm{th}}$ reflection point on $\partial \Omega$ of the broken ray starting at $(z_{-j_{1}}(x,\theta), \theta_{-j_{1}}(x,\theta))$). Then $(h_{j_{2}}-h_{j_{1}}, v_{j_{2}}-v_{j_{1}})$ is the reflection signature from $x$ to $z$, and so we consider 
\begin{equation*}
\widetilde{z} := R_{h_{j_{2}}-h_{j_{1}},v_{j_{2}}-v_{j_{1}}}(z) =  (\widetilde{z}^{1}, \widetilde{z}^{2})
\end{equation*}
as the associated unfolded coordinates of $z$. Up to a rotation of $\Omega$, there are essentially two possibilities depending on whether $\widetilde{z}$ lies on a top or right edge of a reflected copy of the square. We will use $\theta$ to denote an angle value in radians and $\widehat{\theta} = (\cos{\theta}, \sin{\theta})$ to denote its associated direction vector. Furthermore, if $x \in \rn{2}$, we define $x^{\perp} = (-x^{2},x^{1})$. For simplicity of notation, we let $(l_{1},l_{2}) = (h_{j_{2}}-h_{j_{1}}, v_{j_{2}}-v_{j_{1}})$.

We have one of the following equations for $\widetilde{z}$ as a function of the starting point $x$ with direction $\theta$:
\begin{align}
\widetilde{z} & = (\widetilde{z}^{1}, l_{2}) = (x^{1} + (l_{2} - x^{2})\cot{\theta}, l_{2})\\
\widetilde{z} & = (l_{1}, \widetilde{z}^{2}) = (l_{1}, x^{2} + (l_{1} - x^{1})\tan{\theta}).
\end{align}

In the integral expression of the normal operator, we first make the substitution $\eta = S_{l_{1},l_{2}}(\theta)$, which is the reflected angle after applying the billiard map $j_{2}-j_{1}$ times to $(x+\tau_{-}(x,\widehat{\theta})\widehat{\theta},\, \widehat{\theta})$. Note that $|l_{1}|+|l_{2}| = j_{2} - j_{1}$. Obviously, the Jacobian determinant of this transformation is $1$. We also rewrite $z_{j_{2}-j_{1}}(x,\widehat{\theta}) = \pi_{1} \circ T^{j_{2}-j_{1}}(x+\tau_{-}(x,\widehat{\theta})\widehat{\theta}, \widehat{\theta})$ as the function $z(x,\eta)$. In particular, we may have either
\begin{align}
z(x,\eta) & = R_{l_{1},l_{2}}^{-1}\left(x^{1} + (l_{2} - x^{2})\cot\left(S_{l_{1},l_{2}}^{-1}(\eta)\right),\quad l_{2}\right)\label{eq:zxeta1}\\
& = \left((-1)^{l_{1}+1}l_{1} + \frac{1-(-1)^{l_{1}}}{2} + (-1)^{l_{1}}\left[ x^{1} + (l_{2} - x^{2})\cot\left(S_{l_{1},l_{2}}^{-1}(\eta)\right) \right], \quad \frac{1-(-1)^{l_{2}}}{2} \right) \notag
\end{align}
or
\begin{align}
z(x,\eta) & = R_{l_{1},l_{2}}^{-1}\left(l_{1}, \quad x^{2} + (l_{1} - x^{1})\tan\left(S_{l_{1},l_{2}}^{-1}(\eta)\right) \right) \label{eq:zxeta2}\\
& = \left(\frac{1-(-1)^{l_{1}}}{2}, \quad (-1)^{l_{2}+1}l_{2} + \frac{1-(-1)^{l_{2}}}{2} + (-1)^{l_{2}}\left[ x^{2} + (l_{1} - x^{1})\tan\left(S_{l_{1},l_{2}}^{-1}(\eta)\right) \right] \right).\notag
\end{align}

\begin{figure}
\centering
\def \svgwidth{0.7\columnwidth}
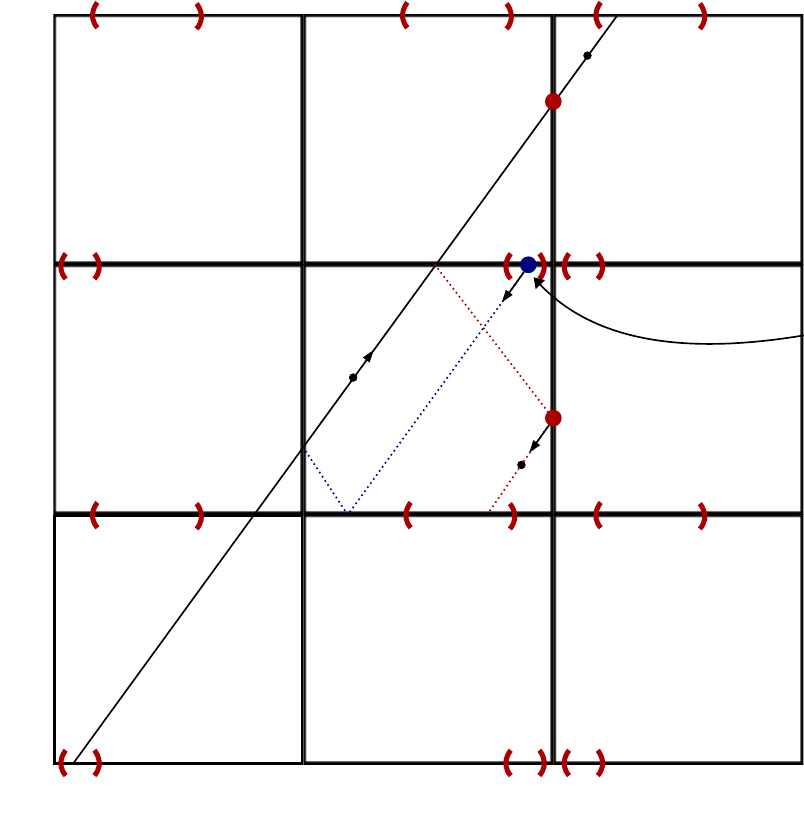
\caption{The basic setup for how we define coordinates in the integral formula for $N_{\sigma, E, \alpha}f$. Here $j_{2} =4$ and $j_{1}=2$. The original copy of $\Omega$ is in the center of the tiling, and $E$ consists of two disjoint open intervals of the boundary.\label{fig:substitution}}
\end{figure}
Thus we have an integrand involving $f(z(x,\eta) + t\widehat{\eta})$, where $\widehat{\eta} = (\cos{\eta},\sin{\eta})$. The next step is to make the substitution $y = z(x,\eta) + t\widehat{\eta}$ (see Figure \ref{fig:substitution}). We compute it's Jacobian determinant to be
\begin{align}
\left| \frac{\partial y}{\partial \eta \partial t} \right| & = |(-1)^{l_{2}}(l_{2}-x^{2})\csc{\eta} + t| = |(l_{2}-x^{2})\csc{\theta} + t| = |\widetilde{y}-\widetilde{z}|\left| \frac{l_{2}-x^{2}}{\widetilde{y}^{2} - \widetilde{z}^{2}} + 1 \right|\notag\\
& = |\widetilde{y}-\widetilde{z}|\left| \frac{\widetilde{y}^{2}-x^{2}}{\widetilde{y}^{2} - l_{2}} \right|
\end{align}
if $z(x,\eta)$ is given by (\ref{eq:zxeta1}). Similarly, we obtain
\begin{equation*}
\left| \frac{\partial y}{\partial \eta \partial t} \right| = |(-1)^{l_{1}}(l_{1} -x^{1})\sec{\eta} + t | = |\widetilde{y}-\widetilde{z}|\left| \frac{\widetilde{y}^{1}-x^{1}}{\widetilde{y}^{1} - l_{1}} \right|
\end{equation*}
when $z(x,\eta)$ is given by (\ref{eq:zxeta2}). Here we've used the fact that $\sec{\eta} = (-1)^{l_{1}}\sec{\theta}$ and $\csc{\eta} = (-1)^{l_{2}} \csc{\theta}$. In order to fully carry out the substitution in the integral, we need to write $z$ as a function of $x$ and $y$. We will need the auxiliary equation
\begin{equation}
(\widetilde{y} - \widetilde{z}) \cdot (\widetilde{z} - x)^{\perp} = 0,
\end{equation}
which is a consequence of the reflection law. In the case where $\widetilde{z}^{2}$ is undetermined, we have $\widetilde{z}^{1} = l_{1}$ and then
\begin{equation}
\widetilde{z}^{2} - \widetilde{y}^{2} = \frac{(\widetilde{y}^{2}-x^{2})(l_{1} - \widetilde{y}^{1})}{\widetilde{y}^{1} - x^{1}}.
\end{equation}
Similarly, if $\widetilde{z}^{1}$ is undetermined then
\begin{equation}
\widetilde{z}^{1} - \widetilde{y}^{1} = \frac{(\widetilde{y}^{1} - x^{1})(l_{2} - \widetilde{y}^{2})}{\widetilde{y}^{2} - x^{2}}.
\end{equation}

In the given case $\widetilde{z}^{1} = l_{1}$, the Jacobian determinant $\left| \frac{\partial y}{\partial \eta \partial t} \right|$ simplifies to
\begin{align*}
 |\widetilde{y}-\widetilde{z}|\left| \frac{\widetilde{y}^{1}-x^{1}}{\widetilde{y}^{1} - l_{1}} \right| & = \sqrt{ (\widetilde{y}^{1} - \widetilde{z}^{1})^{2} + (\widetilde{y}^{2} - x^{2})^{2} \frac{(l_{1} - \widetilde{y}^{1})^{2}}{\widetilde{y}^{1} - x^{1})^{2}}} \left| \frac{\widetilde{y}^{1}-x^{1}}{\widetilde{y}^{1} - l_{1}} \right| \\
 & = |\widetilde{y} - x|.
 \end{align*}
The same holds when $\widetilde{z}^{2} = l_{2}$. Conveniently, $|\widetilde{y}-x|$ is strictly positive for all $x,y \in \mathrm{supp}(f)$ so long as $l_{1}$ or $l_{2}$ is nonzero (i.e. $j_{2} - j_{1} \neq 0$). To see this, just note that for $(l_{1},l_{2}) \neq (0,0)$, we have $\widetilde{y} \notin \Omega$, and thus $|\widetilde{y} - x| \geq d(\mathrm{supp}(f), \partial \Omega) > 0$. Thus the reciprocal factor $\frac{1}{|\widetilde{y}-x|}$ resulting from the change of variables is smooth and bounded as a function of $y$ and $x$. Using again the fact that $\cos{\eta} = (-1)^{l_{1}}\cos{\theta}$, $\sin{\eta} = (-1)^{l_{2}}\sin{\theta}$ we have that
\begin{equation*}
\widehat{\eta} = ((-1)^{l_{1}}\cos{\theta},(-1)^{l_{2}}\sin{\theta}) = \frac{1}{|\widetilde{y}-x|}( (-1)^{l_{1}}(\widetilde{y}^{1}- x^{1}), (-1)^{l_{2}}(\widetilde{y}^{2} - x^{2})).
\end{equation*}
Altogether, we obtain that $N_{\sigma, E,\alpha, reflect}f(x)$ consists of a sum of terms of the form
\begin{equation}
\int_{\rn{2}} |\alpha_{k}(T^{-j_{2}}(y + \tau_{-}(y,\widehat{\eta})\widehat{\eta}, \widehat{\eta}))|^{2} w_{j_{1}}\left( x,\frac{\widetilde{y}-x}{|\widetilde{y}-x|} \right) w_{j_{2}}\left(y, \widehat{\eta}\right)f\left(y\right)\, \frac{dy}{|\widetilde{y} -x|}.\label{eq:normalreflectedterm}
\end{equation}

Note that one must be very careful in the above formulation in order to show that the integrand in (\ref{eq:normalreflectedterm}) is smooth where it is supported. Consider the case if $\alpha_{k}$ is non vanishing for a broken ray which starts at a corner point, but still none of the broken rays in the support of $\alpha$ contain another corner point after some positive number of reflections. Within the factor of (\ref{eq:normalreflectedterm}) given by
\begin{equation*}
|\alpha_{k}(T^{-j_{2}}(y + \tau_{-}(y,\widehat{\eta})\widehat{\eta}, \widehat{\eta}))|^{2},
\end{equation*}
we see that $T^{-j_{2}}(y + \tau_{-}(y,\widehat{\eta})\widehat{\eta}, \widehat{\eta})$ fails to be smooth for $(y, \eta)$ that are pulled back to this corner point. However, as previously discussed, the construction of $\alpha_{k}$ ensures that it is smooth on $\rn{2}\times \mathbb{S}^{1}$. Recall that $B = B\left( \left( \frac{1}{2}, \frac{1}{2} \right), \frac{\sqrt{2}}{2}\right)$ is the ball that circumscribes $\Omega$. By a possible rotation of $\Omega$, we may as well assume that $h_{j},v_{j}$ are all nonnegative. Clearly each sequence $\{h_{j}\}$, $\{v_{j}\}$ is also nondecreasing in this case. Next define the function $q_{j}(y, \eta)$ for $1 \leq j \leq N$ as the \textit{final} intersection point of the line $R_{h_{j},v_{j}}(y) + t \widehat{S_{h_{j},v_{j}}^{-1}(\eta)}$ with $\partial B$ for $t \leq 0$. We also define $q_{0}(y,\eta) = y + \tau_{-,\partial B}(y,\eta) \widehat{\eta}$ where $\tau_{-,\partial B}(y,\widehat{\eta})$ is the negative distance to $\partial B$ from $y$ in the direction $-\widehat{\eta}$, defined for $(y,\widehat{\eta}) \in B \times \mathbb{S}^{1}$. We can then replace $\alpha_{k}(T^{-j_{2}}(y + \tau_{-}(y,\widehat{\eta})\widehat{\eta}, \widehat{\eta}))$ with
\begin{equation}
\alpha_{k}(q_{j_{2}}(y,\widehat{\eta}), \, \widehat{S_{h_{j_{2}}, v_{j_{2}}}^{-1}(\eta)}). \label{eq:smoothalphakterm}
\end{equation}
Clearly, (\ref{eq:smoothalphakterm}) depends smoothly on $y, \eta$.

We observe that the definitions of $w_{j_{1}}, w_{j_{2}}$ also involve $\tau_{-}$, which is not smooth on $\Omega \times \mathbb{S}^{1}$. However, since $\sigma$ is compactly supported inside of $\Omega$ on some smooth subdomain, we can modify $\tau_{-}$ inside of $w_{j_{1}},w_{j_{2}}$ as necessary to be smooth without changing the integrals. In particular, we can use the functions $q_{j}$ to define the weight functions instead of the billiard map $T$ and boundary distance function $\tau_{-}$. Therefore, (\ref{eq:normalreflectedterm}) is a smoothing integral operator.

\subsection{Incorporating Corners\label{sec:corners}}
Recall from the proof of Proposition \ref{thm:microlocalsquare} that by unfolding the broken rays, corner points posed no problem. This meant we didn't have to remove any covectors from $\mathcal{M}'$ whose corresponding broken rays passed through a corner point. It is thus important to also consider broken rays with corners for the normal operator so as to make $\mathcal{M}'$ as big as possible. However, for this we no longer have the option of path unfolding, so it will be necessary to treat reflections more directly.

While the corners of $\partial \Omega$ have no well-defined normal vector, we can make sense of a broken ray that touches a corner point by realizing it as a limit of broken rays that avoid the corner. This is described in detail in the 2-dimensional case in the work of Eskin in \cite{eskin}. Essentially, in 2-dimensions there are two reflection directions at a corner point that can be obtained as a limit of broken rays starting at some fixed point $x \in \Omega$ and with reflection point approaching the given corner. In higher dimensions, one can obtain a cone of possible reflection directions, which is thus a hypersurface. However, for the square (or even the $n$-cube), these limiting rays all conveniently coincide. To see this, let $x\in \Omega$ and let $c$ be a corner point. Define $\theta_{c}$ as the angle between $c-x$ and the horizontal axis. For $\theta$ close to $\theta_{c}$, the broken ray from $x$ in the direction $\widehat{\theta}$ will intersect twice near $c$ and return at an angle $\pi+\theta$ with the positive horizontal direction (i.e. it will return in the direction $-\widehat{\theta}$). Therefore, the limit of such rays shows that $\gamma_{x,\theta_{c}}$ unambiguously reflects back towards $x$. This is also consistent with what one expects from path unfolding.

The following simple lemma is useful in that it significantly limits the number of cases to consider of broken rays with corner points.
\begin{lemma}A regular broken ray starting from $E$ contains at most two corner points, and if it has two, then the first one is its starting point.\end{lemma}
\begin{proof}Suppose we have a broken ray $\gamma_{x_{0},\theta_{0}}$ with two distinct corner points. Since reflections from corners return in the opposite direction from the direction of incidence, any broken ray between them will repeat periodically. Thus, either one of the corner points lies in $E$, or there is a point along the broken ray between them that lies in $E$. But this implies the broken ray would have terminated before hitting both corners.\end{proof}

We will use $c \neq x_{0}$ to refer to a corner in the broken ray if it has one, and we will allow for the possibility that $x_{0}$ itself is a corner point. However, as previously discussed, $x_{0}$ being a corner point does not present any problem with regard to the smoothness of the integrands involved in the expansion of the localized normal operator (see the discussion following (\ref{eq:normalreflectedterm})). As before, let $K \Subset \Omega$ be the set on which we assume $f$ has support. 

So given our broken ray $\gamma_{x_{0},\theta_{0}}$, we assume it passes through a corner point $c \neq x_{0}$ at the $m^{\mathrm{th}}$ reflection, where $1 \leq m \leq N-1$, and $N-1$ is the total number of reflections of the curve. Choose a neighborhood $U$ consisting of $(x,\theta) \in \Gamma_{-}(E)$ that are sufficiently close to $(x_{0},\theta_{0})$ in the following sense. At the $m^{\mathrm{th}}$ reflection, most nearby broken rays to $\gamma_{x_{0},\theta_{0}}$ will undergo a reflection along a very short segment near the corner on which $f$ vanishes and $\sigma = 0$. In particular, such segments do not affect the weights $w_{j}$ and so we will be able to discard them. Thus, we choose $U$ small enough so that at the $m$th reflection, all nearby broken rays to $\gamma_{x_{0},\theta_{0}}$ either pass through the corner $c$ or hit close enough to the corner so as to reflect into a short path that contributes nothing to the integral. Note that for such broken rays from $U$ passing very near to the corner $c$ (but not actually touching it), after the $(m+1)^{\mathrm{th}}$ reflection they will be oriented in a direction opposite their respective directions during the $(m-1)^{\mathrm{th}}$ segment. We also choose $U$ small enough so as not to introduce any new corner points among the included broken rays.

Given such an open set $U$, we let $\mathbbm{1}_{U}(x,\theta)$ be a function on $\rn{2} \times \mathbb{S}^{1}$ which is constant on lines and satisfies
\begin{equation}
\mathbbm{1}_{U}(x,\theta) = \left\{ \begin{array}{cc}
1 & \textrm{ if $(\gamma_{x,\theta}(t), \dot{\gamma}_{x,\theta}(t)$ passes through $U$ and intersects the corner $c$,}\\
0 & \textrm{ otherwise }
\end{array}\right.
\end{equation}
Also let $\alpha(x,\theta)$ be a smooth cutoff function in $\rn{2} \times \mathbb{S}^{1}$ which vanishes outside the set of lines associated with $U$. We write the localized broken ray transform as
\begin{align}
I_{\sigma, E, \alpha}f(x,\theta) & = \sum_{j=0}^{N} \int_{\rn{+}}(1- \mathbbm{1}_{U}(x,\theta))\alpha(x,\theta)[w_{j}f](z_{j}(x,\theta) + t\theta_{j}(x,\theta), \theta_{j}(x,\theta))\,dt \notag\\
& + \sum_{j=0}^{N-1} \int_{\rn{+}}\mathbbm{1}_{U}(x,\theta)\alpha(x,\theta)[w_{j}f](z_{j}(x,\theta) + t\theta_{j}(x,\theta), \theta_{j}(x,\theta))\,dt.
\end{align}
We then have
\begin{align}
I_{\sigma, E, \alpha}^{*}g(x) & = \sum_{j=0}^{N} \int_{\mathbb{S}^{1}}[(1- \mathbbm{1}_{U})\alpha g](z_{-j}(x,\theta), \theta_{-j}(\theta))w_{j}(x,\theta)\,d\theta \notag\\
& + \sum_{j=0}^{N-1} \int_{\mathbb{S}^{1}}[\mathbbm{1}_{U}\alpha g](z_{-}(x,\theta), \theta_{-j}(x,\theta))w_{j}(x,\theta)\,d\theta.
\end{align}
Finally, the normal operator is given by
\begin{align}
I_{\sigma, E, \alpha}^{*}I_{\sigma, E, \alpha}f(x) & = \sum_{j_{1}=0}^{N}\sum_{j_{2}=0}^{N}\int_{\mathbb{S}^{1}}\int_{\rn{+}}[(1-\mathbbm{1}_{U})|\alpha|^{2}](z_{-j_{1}}(x,\theta),\theta_{-j_{1}}(x,\theta))w_{j_{1}}(x,\theta) \notag\\
& \quad \cdot [w_{j_{2}}f](z_{j_{2}-j_{1}}(x,\theta) + t\theta_{j_{2}-j_{1}}(x,\theta), \theta_{j_{2}-j_{1}}(x,\theta))\, dt \,d\theta \notag\\
& +  \sum_{j_{1}=0}^{N-1}\sum_{j_{2}=0}^{N-1}\int_{\mathbb{S}^{1}}\int_{\rn{+}}[\mathbbm{1}_{U}|\alpha|^{2}](z_{-j_{1}}(x,\theta),\theta_{-j_{1}}(x,\theta))w_{j_{1}}(x,\theta) \notag\\
& \quad \cdot [w_{j_{2}}f](z_{j_{2}-j_{1}}(x,\theta) + t\theta_{j_{2}-j_{1}}(x,\theta), \theta_{j_{2}-j_{1}}(x,\theta))\, dt \,d\theta.
\end{align}
Here we've used the fact that $(1-\mathbbm{1}_{U})^{2} = 1 - \mathbbm{1}_{U}$. Then break off the part with $j_{2} = j_{1}$ to obtain
\begin{align*}
N_{\sigma, E, \alpha, ballistic}f(x) & =  \sum_{j=0}^{N}\int_{\mathbb{S}^{1}}\int_{\rn{+}}[(1-\mathbbm{1}_{U})|\alpha|^{2}](z_{-j}(x,\theta),\theta_{-j}(x,\theta))w_{j}(x,\theta) \notag\\
& \quad \cdot [w_{j}f](x + \tau_{-}(x,\theta)\theta+ t\theta, \theta)\, dt \,d\theta \notag\\
& +  \sum_{j=0}^{N-1}\int_{\mathbb{S}^{1}}\int_{\rn{+}}[\mathbbm{1}_{U}|\alpha|^{2}](z_{-j}(x,\theta),\theta_{-j}(x,\theta))w_{j}(x,\theta) \notag\\
& \quad \cdot [w_{j}f](x + \tau_{-}(x,\theta)\theta + t\theta, \theta)\, dt \,d\theta.\notag.
\end{align*}
Now we have to be a bit careful in order to combine both summands. In the first summand, whenever $(z_{-m}(x,\theta), \theta_{-m}(x,\theta))$ is in the support of $(1-\mathbbm{1}_{U})\alpha^{2}$, we know that $x+\tau_{-}(x,\theta)\theta + t\theta$ is a line segment very near the corner $c$ where $f$ vanishes. Thus we can throw out the term involving $j = m$. Now for the terms with $j < m$, we may immediately combine the summands into one, thus eliminating the presence of $\mathbbm{1}_{U}$. That is, we have
\begin{align}
N_{\sigma, E, \alpha, ballistic}f(x) & = \sum_{j=0}^{m-1}\int_{\mathbb{S}^{1}}\int_{\rn{+}}|\alpha|^{2}(z_{-j}(x,\theta),\theta_{-j}(x,\theta))w_{j}(x,\theta) \notag\\
& \quad \cdot [w_{j}f](x + \tau_{-}(x,\theta)\theta+ t\theta, \theta)\, dt \,d\theta \notag\\
& + \sum_{j_{1}=m+1}^{N}\int_{\mathbb{S}^{1}}\int_{\rn{+}}[(1-\mathbbm{1}_{U})|\alpha|^{2}](z_{-j_{1}}(x,\theta),\theta_{-j_{1}}(x,\theta))w_{j_{1}}(x,\theta) \notag\\
& \quad \cdot [w_{j_{1}}f](x + \tau_{-}(x,\theta)\theta+ t\theta, \theta)\, dt \,d\theta \notag\\
& + \sum_{j_{2}=m}^{N-1}\int_{\mathbb{S}^{1}}\int_{\rn{+}}[\mathbbm{1}_{U}|\alpha|^{2}](z_{-j_{2}}(x,\theta),\theta_{-j_{2}}(x,\theta))w_{j_{2}}(x,\theta) \notag\\
& \quad \cdot [w_{j_{2}}f](x + \tau_{-}(x,\theta)\theta + t\theta, \theta)\, dt \,d\theta. \label{eq:normalballistic2}
\end{align}

\begin{figure}
\centering
\def \svgwidth{0.7\columnwidth}
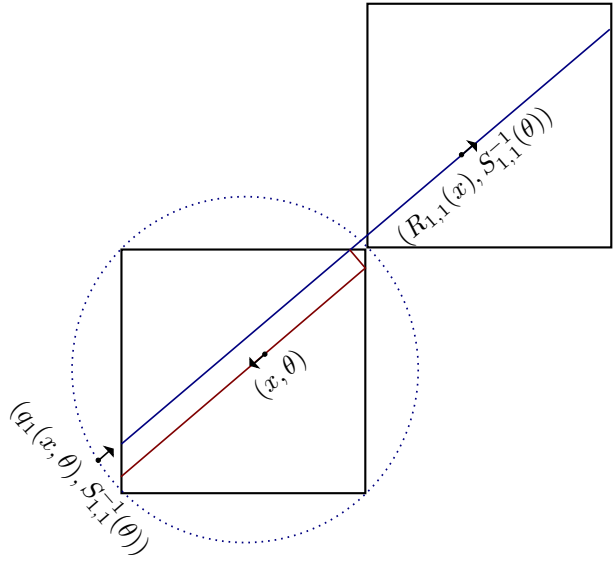
\caption{Here we demonstrate the construction of $q_{1}(x,\theta)$ in the case of a beam of broken rays very close to a corner point.\label{fig:qm}}
\end{figure}

For the second and third terms of (\ref{eq:normalballistic2}), we use the fact that for each $j_{2}$ with $m \leq j_{2} \leq N-1$, $x+(\tau_{-}(x,\theta)+t)\theta$ in the third term has the same reflection signature $(l_{1},l_{2})$ as in the second term, but with $j_{1} = j_{2} + 1$. This has to do with the supports of the corresponding cutoff functions. Fix $j_{1},j_{2}$ such that $m \leq j_{2} = j_{1}-1$, and let $(l_{1},l_{2})$ be the reflection signature of $(x,\theta)$ relative to $(z_{-j_{1}}(x,\theta), \theta_{-j_{1}}(x,\theta))$ in the second term or $(z_{-j_{2}}(x,\theta), \theta_{-j_{2}}(x,\theta))$ in the third. The key idea is that both reflection signatures will be the same. Recall the circumscribed disk $B((\frac{1}{2},\frac{1}{2}), \frac{\sqrt{2}}{2})$. Similar to the case with no corner point, we consider the sequence of reflection signatures $\{(h_{1},v_{1}), (h_{2},v_{2}), \ldots, (h_{N-1}, v_{N-1}) \}$ in the unfolded domain where all of the unfolded broken rays from $U$ touch and over which $f$ has nonzero integral. Thus we've cut out the tiny ``bow-tie" region near the corner $c$, which is why there are only $N-1$ such pairs. Also recall for $0 \leq m \leq N-1$ the functions $q_{m}(x,\theta)$ as the \textit{final} intersection point of the line $R_{h_{m},v_{m}}(x) + tS_{h_{m},v_{m}}^{-1}(\theta)$ with $\partial B$ for $t \leq 0$. Clearly, $q_{m}$ is smooth on the set of $(x,\theta)$ for which the integrand doesn't vanish, since $\partial B$ is smooth (see Figure \ref{fig:qm}). We then have a local regularized version of the weight function $w_{j}$ given by
\begin{align}
w_{reg,j}(x,\theta) & = \exp\left( -\int_{\rn{+}}\sigma(x-\tau\theta, \theta)\,d\tau \right) \notag\\
& \quad  \cdot \prod_{m=1}^{j}\exp\left( -\int_{\rn{+}}\sigma(q_{m}(x,\theta) + \tau S_{h_{m}, v_{m}}^{-1}(\theta), S_{h_{m}, v_{m}}^{-1}(\theta))\,d\tau \right),
\end{align}
for $0 \leq j \leq N-1$. By construction, $w_{reg,j}(x,\theta)$ coincides with $w_{j}(x,\theta)$ or $w_{j-1}(x,\theta)$ on the broken rays we've isolated, with the index depending on whether the broken ray intersects $c$. This allows us to rewrite the ballistic part of the normal operator as
\begin{align*}
& N_{\sigma, E, \alpha, ballistic}f(x) \\
& = \sum_{j=0}^{m-1}\int_{\mathbb{S}^{1}}\int_{\rn{+}}|\alpha|^{2}(z_{-j}(x,\theta),\theta_{-j}(x,\theta))w_{j}(x,\theta)  [w_{j}f](x + \tau_{-}(x,\theta)\theta+ t\theta, \theta)\, dt \,d\theta \notag\\
& + \sum_{j=m}^{N-1}\int_{\mathbb{S}^{1}}\int_{\rn{+}}|\alpha|^{2}(q_{j}(x,\theta), S_{h_{j}, v_{j}}^{-1}(\theta))w_{reg,j}(x,\theta) [w_{reg,j}f](x + \tau_{-}(x,\theta)\theta+ t\theta, \theta)\, dt \,d\theta \notag.
\end{align*}
From here, we can make the change of variables $s = t + \tau_{-}(x,\theta)$. Then from Lemma 2 of \cite{xraygeneric} we have a classical pseudodifferential operator of order $-1$.

Now we need to deal with $N_{\sigma, E, \alpha, reflect}$. However, the approach is similar. Recall we have
\begin{align}
N_{\sigma, E, \alpha, reflect}f(x) & = \sum_{j_{1}=0}^{N}\sum_{j_{2}=0, j_{2}\neq j_{1}}^{N}\int_{\mathbb{S}^{1}}\int_{\rn{+}}[(1-\mathbbm{1}_{U})|\alpha|^{2}](z_{-j_{1}}(x,\theta),\theta_{-j_{1}}(x,\theta))w_{j_{1}}(x,\theta) \notag\\
& \quad \cdot [w_{j_{2}}f](z_{j_{2}-j_{1}}(x,\theta) + t\theta_{j_{2}-j_{1}}(x,\theta), \theta_{j_{2}-j_{1}}(x,\theta))\, dt \,d\theta \notag\\
& +  \sum_{j_{1}=0}^{N-1}\sum_{j_{2}=0, j_{2}\neq j_{1}}^{N-1}\int_{\mathbb{S}^{1}}\int_{\rn{+}}[\mathbbm{1}_{U}|\alpha|^{2}](z_{-j_{1}}(x,\theta),\theta_{-j_{1}}(x,\theta))w_{j_{1}}(x,\theta) \notag\\
& \quad \cdot [w_{j_{2}}f](z_{j_{2}-j_{1}}(x,\theta) + t\theta_{j_{2}-j_{1}}(x,\theta), \theta_{j_{2}-j_{1}}(x,\theta))\, dt \,d\theta. \label{eq:normalreflectcorner1}
\end{align}
Again suppose that the broken ray $\gamma_{x_{0},\theta_{0}}$ intersects a corner $c$ at the $m$th reflection, where $1 \leq m \leq N-1$. We claim that in the first summation above, the terms with $j_{2} = m$ vanish. To see this, note that if $(x,\theta)$ is such that $z_{-j_{1}}(x,\theta),\theta_{-j_{1}}(x,\theta)$ is in the support of $(1-\mathbbm{1}_{U})|\alpha|^{2}$, then reflecting $j_{2} = m$ times from that point will put us in the small ``bow-tie" region near $c$ on which the corresponding line integral vanishes. So we rewrite the first term in (\ref{eq:normalreflectcorner1}) as
\begin{align}
 & \sum_{j_{1}=0}^{N}\sum_{j_{2}=0, j_{2}\neq j_{1}, j_{2}\neq m}^{N}\int_{\mathbb{S}^{1}}\int_{\rn{+}}[(1-\mathbbm{1}_{U})|\alpha|^{2}](z_{-j_{1}}(x,\theta),\theta_{-j_{1}}(x,\theta))w_{j_{1}}(x,\theta) \notag\\
& \quad \cdot [w_{j_{2}}f](z_{j_{2}-j_{1}}(x,\theta) + t\theta_{j_{2}-j_{1}}(x,\theta), \theta_{j_{2}-j_{1}}(x,\theta))\, dt \,d\theta. \label{eq:normalreflectcorner2}
\end{align}
The important thing is that in (\ref{eq:normalreflectcorner2}) the reflection signature of $z_{j_{2}-j_{1}}$ relative to $z_{-j_{1}}$ for each pair $(j_{1},j_{2})$ is the same wherever $\alpha^{2}(1-\mathbbm{1}_{U})$ is supported. 

As in \S \ref{sec:normalsimplify}, for each pair of indices $(j_{1},j_{2})$, we consider the reflection signature when going from $(x,\theta)$ to $(z_{j_{2}-j_{1}}, \theta_{j_{2}-j_{1}})$. This is given by $(h_{j_{2}}- h_{j_{1}}, v_{j_{2}}-v_{j_{1}})$. So as in (\ref{eq:normalreflectedterm}) we define
\begin{align*}
\widetilde{y}_{j_{1},j_{2}} & = R_{(h_{j_{2}}-h_{j_{1}}, v_{j_{2}}-v_{j_{1}})}(y)\\
\widehat{\eta} & = \frac{1}{|\widetilde{y}_{j_{1},j_{2}}-x|}( (-1)^{h_{j_{2}}-h_{j_{1}}}(\widetilde{y}_{j_{1},j_{2}}^{1}- x^{1}), (-1)^{v_{j_{2}}-v_{j_{1}}}(\widetilde{y}_{j_{1},j_{2}}^{2} - x^{2}))\\
\widehat{\theta} & = \frac{\widetilde{y}_{j_{1},j_{2}} - x}{|\widetilde{y}_{j_{1},j_{2}} - x|},
\end{align*}
and we can rewrite
\begin{align}
 & N_{\sigma, E, \alpha, reflect}f(x) \notag\\
 & = \sum_{j_{1}=0}^{N-1}\sum_{j_{2}=0, j_{2}\neq j_{1}}^{N-1}\int_{\rn{2}}\left[(1-\mathbbm{1}_{U})|\alpha|^{2}\right] \left(q_{j_{1}}\left(x,\frac{\widetilde{y}_{j_{1},j_{2}} - x}{|\widetilde{y}_{j_{1},j_{2}} - x|}\right), S_{h_{j_{1}},v_{j_{1}}}^{-1}\left(\frac{\widetilde{y}_{j_{1},j_{2}} - x}{|\widetilde{y}_{j_{1},j_{2}} - x|}\right)\right)\notag\\
 & \quad \cdot w_{reg,j_{1}}\left(x,\frac{\widetilde{y}_{j_{1},j_{2}} - x}{|\widetilde{y}_{j_{1},j_{2}} - x|}\right) \, w_{reg,j_{2}}(y, \widehat{\eta})\, f(y) \, \frac{dy}{|\widetilde{y}_{j_{1},j_{2}} - x|}\notag\\
& +  \sum_{j_{1}=0}^{N-1}\sum_{j_{2}=0, j_{2}\neq j_{1}}^{N-1}\int_{\rn{2}}[\mathbbm{1}_{U}|\alpha|^{2}]\left(q_{j_{1}}\left(x,\frac{\widetilde{y}_{j_{1},j_{2}} - x}{|\widetilde{y}_{j_{1},j_{2}} - x|}\right), S_{h_{j_{1}},v_{j_{1}}}^{-1}\left(\frac{\widetilde{y}_{j_{1},j_{2}} - x}{|\widetilde{y}_{j_{1},j_{2}} - x|}\right)\right)\notag\\
 & \quad \cdot w_{reg,j_{1}}\left(x,\frac{\widetilde{y}_{j_{1},j_{2}} - x}{|\widetilde{y}_{j_{1},j_{2}} - x|}\right) \, w_{reg,j_{2}}(y, \widehat{\eta})\, f(y) \, \frac{dy}{|\widetilde{y}_{j_{1},j_{2}} - x|}\notag\\
 & = \sum_{j_{1}=0}^{N-1}\sum_{j_{2}=0, j_{2}\neq j_{1}}^{N-1}\int_{\rn{2}}\left| \alpha \left(q_{j_{1}}\left(x,\frac{\widetilde{y}_{j_{1},j_{2}} - x}{|\widetilde{y}_{j_{1},j_{2}} - x|}\right), S_{h_{j_{1}},v_{j_{1}}}^{-1}\left(\frac{\widetilde{y}_{j_{1},j_{2}} - x}{|\widetilde{y}_{j_{1},j_{2}} - x|}\right)\right) \right|^{2}\notag\\
 & \quad \cdot w_{reg,j_{1}}\left(x,\frac{\widetilde{y}_{j_{1},j_{2}} - x}{|\widetilde{y}_{j_{1},j_{2}} - x|}\right) \, w_{reg,j_{2}}(y, \widehat{\eta})\, f(y) \, \frac{dy}{|\widetilde{y}_{j_{1},j_{2}} - x|}\label{eq:normalreflectbigkahuna}.
\end{align}

Finally, we are ready to prove the result about the decomposition of $N_{\sigma, E, \alpha}$.
\begin{proof}[Proof of Proposition \ref{prop:normaldecomposition}] Recall that we have already shown in the discussion, see (\ref{eq:normalop}), that
\begin{equation*}
N_{\sigma, E, \alpha} = N_{\sigma, E, \alpha, ballistic} + N_{\sigma, E, \alpha, reflect},
\end{equation*}
where $N_{\sigma, E, \alpha, ballistic}$ is a classical pseudo differential operator of order $-1$ that is elliptic on $\mathcal{M}$ (more generally on the subset of the cotangent space given by $\mathcal{M}'$). So we have a parametrix $Q$ of $N_{\sigma, E,\alpha, ballistic}$ of order $1$ that is elliptic on $\mathcal{M}$.

For the reflected part, we have a sum of terms of the form (\ref{eq:normalreflectedterm}) in the case of a neighborhood of broken rays that don't have any corner points (except possibly as a starting point), or $N_{\sigma, E, \alpha, reflect}$ has the form of (\ref{eq:normalreflectbigkahuna}), both of which have $C_{0}^{\infty}$ Schwartz kernels. Applying $Q$ to both sides of (\ref{eq:normalop}) yields the result.
\end{proof}

\begin{proposition}Under the conditions of Theorem \ref{thm:stabilityestimate}, without assuming that $I_{\sigma, E, \alpha}$ is injective,
\begin{enumerate}
\item[(a)] one has the a priori estimate
\begin{equation*}
\|f\|_{L^{2}(K)} \leq C\|N_{\sigma, E, \alpha}f\|_{H^{1}(\Omega)} + C_{s}\|f\|_{H^{-s}(\Omega)}, \quad \forall s;
\end{equation*}
\item[(b)]  $\mathrm{Ker}{\, I_{\sigma, E, \alpha}}$ is finite dimensional and included in $C^{\infty}(K)$.
\end{enumerate} \label{prop:aprioristability}
\end{proposition}
\begin{proof}The proof is essentially the same as that for (\cite{xraygeneric}, Proposition 3). To prove part (a), by Proposition \ref{prop:normaldecomposition} we have that
\begin{equation*}
QN_{\sigma, E, \alpha}f = f + QN_{\sigma, E, \alpha, reflect}f + S_{1}f =: f + S_{2}f.
\end{equation*}
$S_{2}$ has $C_{0}^{\infty}$ Schwartz kernel $S_{2}(x,y)$, and by elliptic regularity we have
\begin{equation*}
\|f\|_{L^{2}(K)} - C_{s}\|f\|_{H^{-s}(\Omega)} \leq \|QN_{\sigma, E, \alpha}f\|_{L^{2}(K)} \leq C \|f\|_{H^{1}(\Omega)}
\end{equation*}
as desired.

For part (b), if $f \in \mathrm{Ker}{\, I_{\sigma, E, \alpha} }$, then $(\mathrm{Id} + S_{2})f = 0$, and also $S_{2}$ is a compact operator on $L^{2}(K)$ with smooth Schwartz kernel. Clearly $S_{2}f \in L^{2}$ if $f \in \mathcal{E}'(\Omega)$. Furthermore, for each $k \in \mathbb{N}$, we can estimate
\begin{equation}
\|f\|_{C^{k}(K)}= \|S_{2}f\|_{C^{k}(K)} \lesssim \mathrm{sup}_{y\in \Omega}\left( \|S_{2}(\cdot, y)\|_{C^{k}(K)} \right) \|f\|_{L^{2}(\Omega)}.
\end{equation}
Therefore, $f \in C^{\infty}(\Omega)$. Finally, the fact that $I+S_{2}$ is Fredholm implies that its kernel is finite dimensional. This in turn implies that $\mathrm{Ker}(I_{\sigma, E, \alpha})$ is finite dimensional.\end{proof}

\section{Reducing the Smoothness Requirement on $\sigma$, $\alpha$ \label{sec:reducingsmoothness}}
Given a choice of smooth cutoff $\alpha$ and a smooth $\sigma$ vanishing near $\partial \Omega$ and such that $I_{\sigma, E, \alpha}$ is injective, we would like to be able to perturb $\alpha$ and $\sigma$ in $C^{2}$ slightly and still have $N_{\sigma, E,\alpha}$ be injective. We do this according to the following modified version of (\cite{xraygeneric}, Proposition 4).
\begin{proposition}Assume that $\sigma, \alpha$ are fixed and belong to $C^{2}$. Let $(\sigma', \alpha')$ be $O(\delta)$ close to $(\sigma, \alpha)$ in $C^{2}$. Then there exists a constant $C > 0$ that depends on an a priori bound on the $C^{2}$ norm of $(\sigma, \alpha)$ such that
\begin{equation}
\left\| (N_{\sigma', E, \alpha'} - N_{\sigma, E, \alpha})f \right\|_{H^{1}(\Omega)} \leq C \delta \|f\|_{L^{2}(K)}. 
\end{equation}\label{prop:normalperturb}
\end{proposition}

\begin{proof}
Recall that $N_{\sigma,E,\alpha} = N_{\sigma, E, \alpha, ballistic} + N_{\sigma, E, \alpha, reflect}$, where $N_{\sigma, E, \alpha, ballistic}$ is a weakly singular integral operator with smooth kernel and $N_{\sigma, E, \alpha, reflect}$ is smoothing. By the discussion in the proof of Proposition 4 of \cite{xraygeneric}, we have 
\begin{equation}
\|N_{\sigma, E, \alpha, ballistic} - N_{\sigma', E, \alpha', ballistic}\|_{H^{1}(\Omega)} \leq C \delta \|f\|_{L^{2}(K)}.
\end{equation}
Similarly, the Schwartz kernels of $N_{\sigma, E, \alpha, reflect}$ depend smoothly on $\sigma$ and $\alpha$ (see (\ref{eq:normalreflectedterm}) and (\ref{eq:normalreflectbigkahuna})), so the same estimate holds for this term. The completes the proof.
\end{proof}

\begin{proof}[Proof of Theorem \ref{thm:stabilityestimate}] First note that if $I_{\sigma, E, \alpha}: L^{2}(K) \to L^{2}(\Gamma_{-})$ is injective, then $N_{\sigma, E, \alpha}:L^{2}(K)\to H^{1}(\Omega)$ is also injective by an integration by parts. Next, since $N_{\sigma, E, \alpha, ballistic}$ is elliptic on $\mathcal{M}$, we may apply a parametrix $Q$ to get the equation $QN_{\sigma, E, \alpha}f = f + S_{1}f + QN_{\sigma, E, \alpha, reflect}f$. Note that $S_{1} + QN_{\sigma, E, \alpha, reflect}$ is compact. So by elliptic regularity we have the estimate
\begin{equation*}
\|f\|_{L^{2}(K)} \leq C( \| N_{\sigma, E, \alpha}f\|_{H^{1}(\Omega)}  + \|(S_{1} + QN_{\sigma, E, \alpha, reflect})f\|_{L^{2}(\Omega)}).
\end{equation*}
By Lemma 2 of \cite{stefanov1}, the compactness of the second term on the right implies that
\begin{equation*}
\|f\|_{L^{2}(K)} \leq C'\|N_{\sigma, E, \alpha}f\|_{H^{1}(\Omega)}.
\end{equation*}
The other inequality is obvious.

Finally, to prove part (b), we just use Proposition \ref{prop:normalperturb} to get
\begin{align*}
\|f\|_{L^{2}(K)} & \leq C \|N_{\sigma_{0}, E, \alpha^{0}}f \|\\
& \leq C\|N_{\sigma_{0}, E, \alpha^{0}}f\|_{H^{1}(\Omega)} - C\|(N_{\sigma_{0},E,\alpha^{0}} - N_{\sigma, E, \alpha})f\|_{H^{1}(\Omega)}\\
& \leq  C\|N_{\sigma_{0}, E, \alpha^{0}}f\|_{H^{1}(\Omega)} - C^{2}\delta \|f\|_{L^{2}(K)}.
\end{align*}
\end{proof}

\section{Discussion}
In this work, we have utilized the machinery of \cite{xraygeneric} for the X-ray transform with generic weights and applied it to the given broken ray transform on the Euclidean square. The technique of path unfolding can be applied to any polyhedral domain, however the square is very nice in that reflections are well-defined even at corner points. The unfolding approach also placed significant restraints on the attenuation coefficient $\sigma$, which is why we could only reasonably tackle the case $\sigma = 0$ and then use the normal operator computation to extend the results to $C_{0}^{2}$ perturbations. It is perhaps of interest to carry out similar computations for the broken ray transform on higher dimensional polyhedral domains, or domains where $\partial \Omega \setminus E$ is flat.  One interesting problem to consider would be to try to get analogous microlocal results for the unit ball. However, one potential obstacle for strictly convex smooth domains is that a beam of collimated rays, upon reflection becomes uncollimated. This makes it difficult to obtain a nice integral form of the normal operator. Regardless, it is the author's hope that more progress can be made on this particular problem and its variants in the future.

\nocite{*}
\bibliographystyle{alpha}
\bibliography{reflectedraynew}

\end{document}